\documentclass[10pt,oneside,a4paper]{article}
\usepackage{mathrsfs}
\usepackage{mathdots}
\usepackage{amsmath}
\usepackage{amsthm}
\usepackage{bbm}
\usepackage{subfloat}
\usepackage{subfigure}
\usepackage{latexsym}
\usepackage{caption2}
\usepackage{algorithmic,algorithm}
\usepackage{fancyhdr,graphicx}
\usepackage{amsfonts}
\usepackage{amssymb}
\usepackage{latexsym,bm}
\usepackage{newlfont}
\usepackage{latexsym}
\usepackage{caption2}
\usepackage{algorithmic,algorithm}
\usepackage{graphicx,subfigure}
\usepackage{amsmath,amssymb}
\usepackage[mathscr]{eucal}
\usepackage{multirow}
\usepackage[pagewise]{lineno}
\usepackage{slashbox}
\usepackage{booktabs}
\usepackage{float}
\usepackage{array}
\usepackage{tabularx}
\newtheorem{theorem}{Theorem}

\newtheorem{example}{Example}
\newtheorem{remark}{Remark}

\numberwithin{figure}{section} \numberwithin{equation}{section}

\makeatletter
\long\def\@makecaption#1#2{%
 \vskip\abovecaptionskip
  \sbox\@tempboxa{{#1.}\quad #2}%
 \ifdim \wd\@tempboxa >\hsize
    { #1.}\quad #2\par
     \else
  \global \@minipagefalse
   \hb@xt@\hsize{\hfil\box\@tempboxa\hfil}%
   \fi
   \vskip\belowcaptionskip}
\makeatother

\title{ \textbf{Fast predictor-corrector approach for the
tempered fractional ordinary differential equations
}}

\author{Jingwei Deng$^{a,b,}$\footnote{E-mail addresses: djw$_-$1205@163.com.},~   Lijing Zhao$^{a,}$\footnote{E-mail addresses: zhaolj10@lzu.edu.cn.},~   Yujiang Wu$^{a,}$\footnote{E-mail addresses: myjaw@lzu.edu.cn.}\\ \vspace{-0.15cm}
\small{$^{a}$ School of Mathematics and Statistics, Gansu Key Laboratory of Applied Mathematics and Complex Systems,
}\\
\small{Lanzhou University, Lanzhou 730000, P.R. China} \\
\small{$^b$ School of Mathematics and Computer Science, Northwest University for Nationalities, Lanzhou 730030, P.R. China}\\
        \vspace{-0.1cm}
        }
\date{}
\begin{document}
\maketitle \makeatletter
\newcommand{\rmnum}[1]{\romannumeral #1} 　　
\newcommand{\Rmnum}[1]{\expandafter\@slowromancap\romannumeral #1@}
\makeatother \vspace{-1cm}
\begin{abstract}
The tempered evolution equation describes the trapped dynamics, widely appearing in nature, e.g., the motion of living particles in viscous liquid.  This paper proposes the fast predictor-corrector approach for the tempered fractional ordinary differential equations by digging out the potential `very' short memory principle. The algorithms basing on the idea of equidistributing are detailedly described; their effectiveness and low computation cost, being linearly increasing with time $t$, are numerically demonstrated.

\end{abstract}



\section{Introduction}
The tempered fractional calculus is a mathematical tool to describe the transition between normal and anomalous diffusions (or the anomalous motion in finite time or bounded physical space). The motion of a L\'{e}vy flight particle can be characterized by the continuous time random walk (CTRW) model with a jump distribution function $\phi(x)\sim |x|^{-(1+\alpha)} (1 < \alpha<2)$ and exponential waiting time distribution.
The stable L\'{e}vy measure for particle displacement makes arbitrarily
large jumps possible and then results in divergent spatial moments \cite{Magdziarz:07}. However, the infinite spatial moments
are not always appropriate for the physical processes; e.g., see \cite{Baeumera:10,Cartea:07a, Cartea:07}. Exponentially tempering the L\'{e}vy measure is a popular way
to make the moments of L\'{e}vy  distributions finite 
in transport models. Then we get the spatially
tempered fractional Fokker-Planck equation \cite{Cartea:07, Sabzikar:15}. This paper focuses on the time tempered fractional derivative, which appears in the Fokker-Planck equation being derived from the CTRW model with tempered power law waiting time distribution \cite{Gajda:11,Meerschaert:09,Schmidt:10}. The tempered power law waiting time measure has finite first moment and makes the trapped dynamics more physical; since sometimes it is necessary to make the first moment of the waiting time measure finite, e.g., the biological particles moving in viscous cytoplasm and displaying trapped dynamical behavior just have finite lifetime.
%
The time tempered dynamics describes the phenomena of the coexistence/transition of subdiffusion and normal diffusion (or the subdiffusion in finite time) which was empirically confirmed in a number of systems \cite{Meerschaert:09,Meerschaert:14}.
More applications for the tempered fractional derivatives and tempered differential equations can be found, for instance, in  poroelasticity \cite{Hanyga:2001},
finance \cite{Cartea:07a}, ground water hydrology \cite{Meerschaert:09},
and geophysical flows \cite{Meerschaert:14}.

Tempered fractional calculus is the generalization of fractional calculus.
Comparing with the classical differential equations, one of the big
challenges we have to face when solving the (tempered) fractional equations is the expensiveness of its computation 
cost besides its complexity, this is because fractional operators
are pseudodifferential operators which are non-local
\cite{Samko:93}. This paper focuses on providing the fast predictor-corrector 
approach for the following tempered fractional ordinary differential equation
\begin{equation} \label{e1dot1}
_{0}^{C}D_{t}^{\alpha,\lambda} x(t)=f(t,x(t)),~~~0<t<T,
\end{equation}
with the initial conditions
\begin{equation} \label{e1dot2}
\frac{d^k}{d t^k}\left(e^{\lambda t}x(t)\right)\big|_{t=0}=c_k,\, k=0,1,\cdots, \lceil \alpha \rceil-1,
\end{equation}
where $\alpha \in (0,\infty)$, $\lambda>0$, $c_k$ are arbitrary real
numbers, and $_{0}^{C}D_{t}^{\alpha,\lambda}$ denotes the tempered fractional derivative in
the Caputo sense \cite{Li:14,Samko:93}, defined by
\begin{equation}\label{e1dot3}
      {_{a}^{C}D}_t^{\alpha,\lambda}u(t)=e^{-\lambda t}~{_{a}^{C}D}_t^{\alpha}\big(e^{\lambda t}u(t)\big)=\frac{e^{-\lambda t}}{\Gamma(n-\alpha)}\int_{a}^t\frac{1}{(t-s)^{\alpha-n+1}}\frac{d^n (e^{\lambda s}u(s))}{d
          s^n}ds,
    \end{equation}
where ${_{a}^{C}D}_t^{\alpha}$  denotes
the Caputo fractional derivative \cite{Podlubny:99}.
Note that when $\lambda=0$, the Caputo tempered fractional derivative
reduces to the Caputo fractional derivative.

One of the effective and popular methods for numerically solving the initial value problem
(\ref{e1dot1})-(\ref{e1dot2}) with $\lambda=0$ is the predictor-corrector approach, raised by
Diethelm et al \cite{Diethelm:02a,Diethelm:02b,Diethelm:04}. This
method is improved in the Preliminary and Appendix sections of \cite{Deng:07d} (or see the review article \cite{Deng:12}), where around half of
the computation cost seems to be cut off and the convergence order is
improved from $\min\{1+\alpha,2 \}$ to
$\min\{1+2\alpha,2 \}$. Some efforts have been made to reduce the computation cost by using the nested meshes basing on the fixed memory principle \cite{Podlubny:99} and the
short memory principle \cite{Ford:01} of fractional operator when
$\alpha \in (0,1)$ in (\ref{e1dot1}), and the short memory principle
is apprehended from a new point of view in \cite{Deng:07c} which extends its effective 
range to $\alpha \in (0,2)$. With the nested meshes, the computation cost can be reduced from $O(h^{-2})$ to
$O(h^{-1} \log(h^{-1}))$ while not losing the numerical accuracy, but its seems that this is more theoretical claims rather than numerical practices; in fact the numerical simulation results in
\cite{Deng:07c,Ford:01} show this. More recently, the so called Jacobian-predictor-corrector approach, is introduced by Zhao and Deng \cite{Zhao:14}, in which the accuracy is greatly improved  while the computation cost is not increased. 

In this paper, we dig out the potential of the short memory
principle of tempered fractional operators and apply it felicitously  to
reduce the computation cost by using the idea of equidistributing
meshes \cite{White:79} for numerically solving the initial value
problem (\ref{e1dot1})-(\ref{e1dot2}) with the predictor-corrector method. In the next section, we
introduce the techniques of equidistributing meshes, present
the detailed numerical schemes, and describe the algorithms by pseudo codes. Numerical examples to illustrate the
efficacy of the algorithm are given in Section \ref{sec:3}. We conclude the paper with some remarks in the last section.

\section{Predictor-corrector Algorithms with Equidistributing Meshes}\label{sec:2}
The initial value problem (\ref{e1dot1})-(\ref{e1dot2}) is
equivalent to the following Volterra integral equation
\cite{Li:14}
\begin{eqnarray} \label{MainEq}
x(t)&=&\sum\limits_{k=0}^{\lceil \alpha \rceil-1} c_{k}
\frac{e^{-\lambda t}t^k}{k!}+\frac{1}{\Gamma(\alpha)}
\int_0^t e^{-\lambda (t-\tau)}(t-\tau)^{\alpha-1}f(\tau,x(\tau)) d \tau
\nonumber\\
&:=&x_0(t)+\frac{1}{\Gamma(\alpha)}
\int_0^t e^{-\lambda (t-\tau)}(t-\tau)^{\alpha-1}f(\tau,x(\tau)) d \tau,
\end{eqnarray}
where $x_0(t):=\sum\limits_{k=0}^{\lceil \alpha \rceil-1} c_k
\frac{e^{-\lambda t} t^k}{k!}$. For uniform nodes $t_{n+1}=(n+1)h,~n=0,1,\cdots,N$, with $h=T/N$ being the steplength of numerical
computation, the above equation can be recast as \cite{Deng:07c,Deng:07d}
\begin{equation} \label{e2dot2}
\begin{array}{lll}
x(t_{n+1})&=& \displaystyle x_0(t_{n+1})
+\frac{1}{\Gamma(\alpha)}
\int_{t_n}^{t_{n+1}}e^{-\lambda (t_{n+1}-\tau)}(t_{n+1}-\tau)^{\alpha-1}f(\tau,x(\tau)) d \tau \\
\\
&&  \displaystyle +\frac{1}{\Gamma(\alpha)}
\int_0^{t_n}e^{-\lambda (t_{n+1}-\tau)}(t_{n+1}-\tau)^{\alpha-1}f(\tau,x(\tau)) d \tau,
\end{array}
\end{equation}
or
\begin{equation} \label{e2dot3}
\begin{array}{lll}
x(t_{n+1})&=&  \displaystyle x(t_n)+x_0(t_{n+1})-x_0(t_n)
\\
\\
&& \displaystyle +\frac{1}{\Gamma(\alpha)}
\int_{t_n}^{t_{n+1}}e^{-\lambda (t_{n+1}-\tau)}(t_{n+1}-\tau)^{\alpha-1}f(\tau,x(\tau)) d \tau \\
\\
&&  \displaystyle +\frac{1}{\Gamma(\alpha)}
\int_0^{t_n}\left(e^{-\lambda (t_{n+1}-\tau)}(t_{n+1}-\tau)^{\alpha-1}-e^{-\lambda (t_n-\tau)}(t_n-\tau)^{\alpha-1}\right)f(\tau,x(\tau))
d \tau.
\end{array}
\end{equation}

For the single step integral
$\int_{t_n}^{t_{n+1}}e^{-\lambda (t-\tau)}(t_{n+1}-\tau)^{\alpha-1}g(\tau) d \tau$, we can employ the rectangle or trapezoidal quadrature formula to approximate it, i.e.,
\begin{equation} \label{e2dot4}
\begin{array}{lll}
&&\int_{t_n}^{t_{n+1}}e^{-\lambda (t_{n+1}-\tau)}(t_{n+1}-\tau)^{\alpha-1}g(\tau) d \tau
\\
&\approx&\int_{t_n}^{t_{n+1}}e^{-\lambda (t_{n+1}-t_n)}(t_{n+1}-\tau)^{\alpha-1}g(t_n) d \tau
=\frac{h^\alpha}{\alpha}e^{-\lambda h}g(t_n),
\end{array}
\end{equation}
or
\begin{equation} \label{e2dot5}
\begin{array}{lll}
&&\int_{t_n}^{t_{n+1}}e^{-\lambda (t_{n+1}-\tau)}(t_{n+1}-\tau)^{\alpha-1}g(\tau) d \tau
\\
&\approx&\int_{t_n}^{t_{n+1}}(t_{n+1}-\tau)^{\alpha-1}
\frac{g(t_{n+1})(\tau-t_n)+e^{-\lambda h}g(t_n)(t_{n+1}-\tau)}{h} d \tau
\\
&=&\frac{h^\alpha}{\alpha(\alpha+1)}\left[\alpha e^{-\lambda h}g(t_n)+g(t_{n+1})\right].
\end{array}
\end{equation}
So, obviously, for the numerical approximations of (\ref{e2dot2}) and
(\ref{e2dot3}), the main computational cost comes from the term
$\int_0^{t_n}\cdot \,d \tau$. Luckily when $t_{n+1} \rightarrow
+\infty$, $(t_{n+1}-\tau)^{\alpha-1}$ decays algebraically with the order
$1-\alpha$ for $\alpha \in (0,1]$, and $(t_{n+1}-\tau)^{\alpha-1}-(t_n-\tau)^{\alpha-1}$ decays with the
order $2-\alpha$ for $\alpha \in (0,2)$, these are the so-called short memory principle
\cite{Ford:01} and the one apprehended from a new point of view
\cite{Deng:07c}; both the integral kernels $\left(e^{-\lambda (t_{n+1}-\tau)}(t_{n+1}-\tau)^{\alpha-1}\right)$ and $\left(e^{-\lambda (t_{n+1}-\tau)}(t_{n+1}-\tau)^{\alpha-1}-e^{-\lambda (t_n-\tau)}(t_n-\tau)^{\alpha-1}\right)$ exponentially decay.  We felicitously apply these decay properties to
reduce the computation cost based on equidistributing meshes rather
than roughly using the nested meshes. The linearly increasing computation cost with time $t$ is shown in the following simulations.
\subsection{Algorithms for (\ref{e2dot2}) when $\alpha \in (0,1]$}

In this subsection, we design numerical schemes for $\alpha \in (0,1]$ by using (\ref{e2dot2}) with the kernel $e^{-\lambda (t_{n+1}-\tau)}(t_{n+1}-\tau)^{\alpha-1}$. As mentioned above, how to compute the integral
$\int_0^{t_n}e^{-\lambda (t_{n+1}-\tau)}(t_{n+1}-\tau)^{\alpha-1}f(\tau,x(\tau)) d \tau$ efficiently is the key of
reducing the computation cost to obtain the numerical approximations to $x(t_{n+1})$.
In the previous predictor-corrector work introduced in \cite{Diethelm:02a,Diethelm:02b,Diethelm:04}, as well as its improved version in \cite{Deng:07d,Deng:12}, $n$ steps' calculations are used to approximate the integral $\int_0^{t_n}\cdot \,d \tau$.
While, by noticing that $(t_{n+1}-\tau)^{\alpha-1}$ decays with power $1-\alpha$, and $e^{-\lambda (t_{n+1}-\tau)}$ damps exponentially (see figures below), we can actually select less mesh points, say
\begin{equation} \label{e2dot6}
0=\tau_{0,n}<\tau_{1,n}<\cdots<\tau_{m_n,n}=t_n,
\end{equation}
at which to approximate the term $\int_0^{t_n}\cdot \,d \tau$ by compounding two-point trapezoidal quadrature formula. In the following, we list two ways of choosing the mesh points $\{\tau_{i,n}\}$.

\subsubsection{Two equidistributing options for selecting the quadrature nodes}

Assuming that we have already get the points $\tau_{i,n}$, $0\leq i<m_n$. For not losing the accuracy but reducing the computation cost, the
first way of selecting the next point $\tau_{i+1,n}$ is based on the
principle that the values of the function $y(\tau)=e^{-\lambda (t_{n+1}-\tau)}(t_{n+1}-\tau)^{\alpha-1}$ are equally distributed, i.e.,
\begin{eqnarray}\label{e2dot7_add}
&& y(\tilde{\tau}_{i+1,n})-y({\tau}_{i,n})
\nonumber\\
&=&e^{-\lambda (t_{n+1}-\tilde{\tau}_{i+1,n})}(t_{n+1}-\tilde{\tau}_{i+1,n})^{\alpha-1}
-e^{-\lambda (t_{n+1}-\tau_{i,n})}(t_{n+1}-\tau_{i,n})^{\alpha-1}
\nonumber\\
&=&e^{-\lambda \left(t_{n+1}-\tau_{i,n}-(\tilde{\tau}_{i+1,n}-\tau_{i,n})\right)}
\left(t_{n+1}-\tau_{i,n}-(\tilde{\tau}_{i+1,n}-\tau_{i,n})\right)^{\alpha-1}
-e^{-\lambda (t_{n+1}-\tau_{i,n})}(t_{n+1}-\tau_{i,n})^{\alpha-1}
\nonumber\\
&=&\Delta y,
\end{eqnarray}
where $\Delta y$ is a given small positive real number. Since (\ref{e2dot7_add}) is a nonlinear equation w.r.t.  $\tilde{\tau}_{i+1,n}$, we use its linear part instead to select the wanted point, and add another condition that promise  $\tau_{i+1,n}$ is at least one step away from $\tau_i$. That is, let 
\begin{equation}\label{e2dot7}
\tilde{\tau}_{i+1,n}=\max\left\{
\begin{array}{l}
\textrm{solve}(\tilde{\tau}_{i+1,n}-\tau_{i,n}=h, \tilde{\tau}_{i+1,n}),
\\
\textrm{solve}\big((\tilde{\tau}_{i+1,n}-\tau_{i,n})\left[(1-\alpha)+\lambda(t_{n+1}-\tau_{i,n})\right]
=\Delta y e^{\lambda(t_{n+1}-\tau_{i,n})(t_{n+1-\tau_{i,n}})},  \tilde{\tau}_{i+1,n} \big)
\end{array}
\right\},
\end{equation}
where `solve(equ, var)' means the solution of `equ' with unknown variable `var'. Secondly, to avoid involving non-equal divided nodes, we take
\begin{equation}\label{e2dot8}
\tau_{i+1,n}=\left\lfloor\frac{\tilde{\tau}_{i+1,n}}{h}\right\rfloor\cdot h.
\end{equation}
Figs. \ref{fig:2.1}-\ref{fig:2.2} and Algorithm \ref{alg:2.1} illustrate this equal-height distribution method more clearly.
\begin{figure}[!htbp]
\includegraphics[scale=0.4]{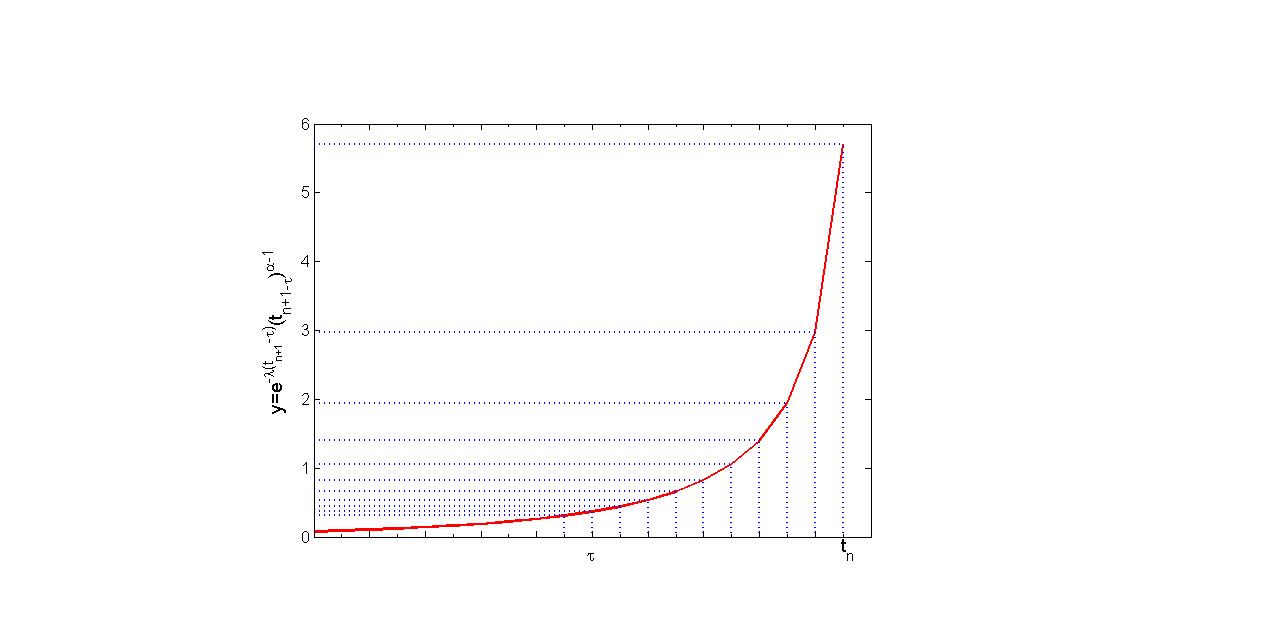}\\
\caption{The relationship between the selected nodes $\{\tau_i\}$ and the function $y=e^{-\lambda(t_{n+1}-\tau)}(t_{n+1}-\tau)^{\alpha-1}$ in the algorithm of equal-height distribution, where $\Delta y=h=1/10$, $t_{n+1}=2$, $\lambda=1$, and $\alpha=0.2$.}\label{fig:2.1}
\end{figure}
\begin{figure}[!htbp]
\includegraphics[scale=0.4]{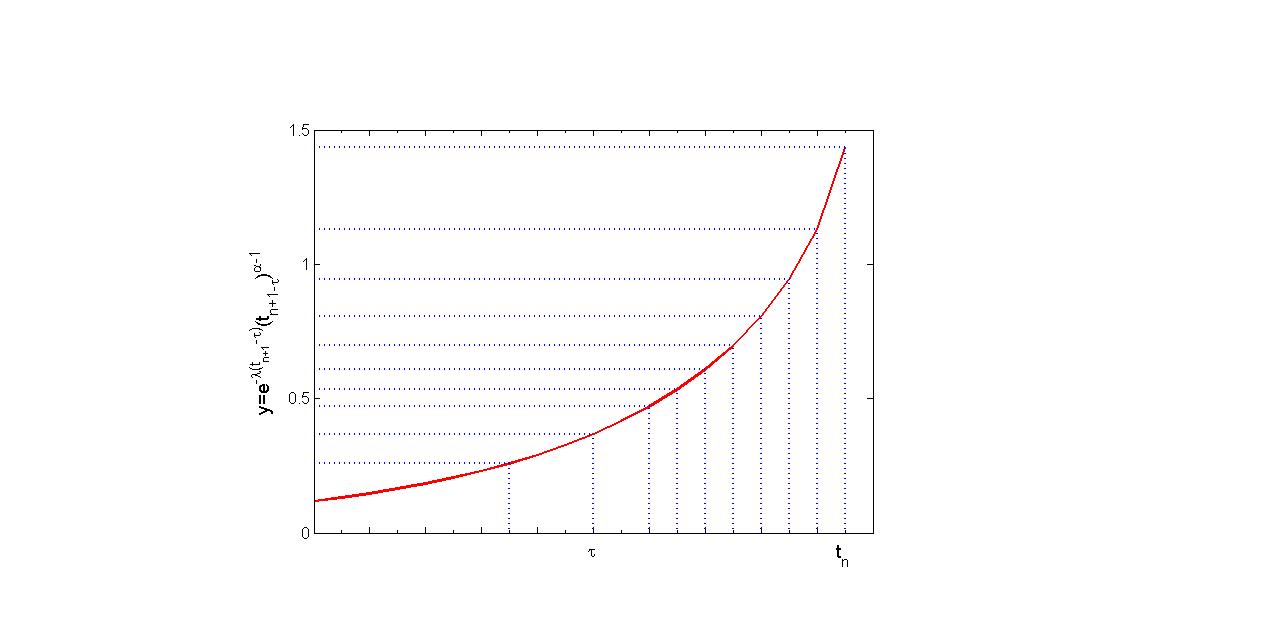}\\
\caption{The relationship between the selected nodes $\{\tau_i\}$ and the function $y=e^{-\lambda(t_{n+1}-\tau)}(t_{n+1}-\tau)^{\alpha-1}$ in the algorithm of equal-height distribution, where $\Delta y=h=1/10$, $t_{n+1}=2$, $\lambda=1$, and $\alpha=0.8$.}\label{fig:2.2}
\end{figure}
\begin{algorithm}[h]
\caption{The pseudocode of equal-height distribution algorithm for (\ref{e2dot7})-(\ref{e2dot8})}\label{alg:2.1}
\begin{algorithmic}[l]
\STATE $i=0$; $\tau_{i,n}=0$; \% $\tau_{0,n}=0;$\\
\STATE $\tau_c=0$; \% \emph{current node is} $\tau_{0,n}$
\WHILE {$\tau_{c}<t_n$}
    \STATE $\tau_{i+1,n}=\tau_{c}+\frac{\Delta y e^{\lambda(t_{n+1}-\tau_{c})(t_{n+1-\tau_{c}})}}
           {\left[(1-\alpha)+\lambda(t_{n+1}-\tau_{c})\right]}$;
    \IF {$\tau_{i+1,n}>t_n$}
         \STATE $\tau_{i+1,n}=t_n$; 
         \STATE break;
    \ENDIF\\
    \% \emph{if} $y=e^{-\lambda(t_{n+1}-\tau)}(t_{n+1}-\tau)^{\alpha-1}$
        \emph{changes too fast}
    \IF{$\tau_{i+1,n}-\tau_{c}<h$}
        \STATE  $\tau_{i+1,n}=\tau_{c}+h$;
        \% \emph{make} $\tau_{i+1,n}$ \emph{be one step away from} $\tau_{i,n}$
    \ELSE
        \STATE  $\tau_{i+1,n}=\lfloor \tau_{i+1,n}/h \rfloor\ast h$;
                \% \emph{let} $\tau_{i+1,n}$ \emph{belong to the uniform mesh points} $\{t_j\}_{j=0}^n$
    \ENDIF
    \STATE  $\tau_{c}=\tau_{i+1,n}$; $i=i+1$;
\ENDWHILE
\end{algorithmic}
\end{algorithm}

The second way of choosing the mesh points $\{\tau_{i,n}\}$ is to make the integrations of $y(\tau)=e^{-\lambda(t_{n+1}-\tau)}(t_{n+1}-\tau)^{\alpha-1}$ w.r.t.  $\tau$ on any interval are almost the same, i.e.,
\begin{equation}\label{e2dot10_add1}
\int_{\tau_{i,n}}^{\tilde{\tau}_{i+1,n}}e^{-\lambda(t_{n+1}-\tau)}
(t_{n+1}-\tau)^{\alpha-1} d\tau=\Delta s,
\end{equation}
where $\Delta s$ is a given small positive real number. Also, since (\ref{e2dot10_add1}) is a nonlinear function w.r.t. $\tilde{\tau}_{i+1,n}$, we simply approximate it by
\begin{eqnarray}\label{e2dot10_add2}
&&\int_{\tau_{i,n}}^{\tilde{\tau}_{i+1,n}}e^{-\lambda(t_{n+1}-\tau)}
(t_{n+1}-\tau)^{\alpha-1} d\tau
\nonumber\\
&\approx& (t_{n+1}-\tau_{i,n})^{\alpha-1}
\int_{\tau_{i,n}}^{\tilde{\tau}_{i+1,n}}e^{-\lambda(t_{n+1}-\tau)} d \tau
=\Delta s.
\end{eqnarray}
Thus, 
\begin{equation}\label{e2dot10}
\tilde{\tau}_{i+1,n}=\max\left\{
\begin{array}{l}
\textrm{solve}(\tilde{\tau}_{i+1,n}-\tau_{i,n}=h, \tilde{\tau}_{i+1,n}),
\\
\textrm{solve}\left((\tilde{\tau}_{i+1,n}-\tau_{i,n})=\Delta s(t_{n+1}-\tau_{i,n})^{1-\alpha}e^{\lambda(t_{n+1}-\tau_{i,n})}, \tilde{\tau}_{i+1,n} \right)
\end{array}
\right\};
\end{equation}
 and then by taking
\begin{equation}\label{e2dot11}
\tau_{i+1,n}=\left\lfloor\frac{\tilde{\tau}_{i+1,n}}{h}\right\rfloor\cdot h,
\end{equation}
$\tau_{i+1,n}$ belongs to the set of the uniform grid points $\{t_j\}_{j=0}^{n}$.
Figs. \ref{fig:2.3}-\ref{fig:2.4} and Algorithm \ref{alg:2.2} show this equal-area distribution criterion of selecting the mesh point $\{\tau_{i,n}\}$ more concretely.
\begin{figure}[!htbp]
\includegraphics[scale=0.4]{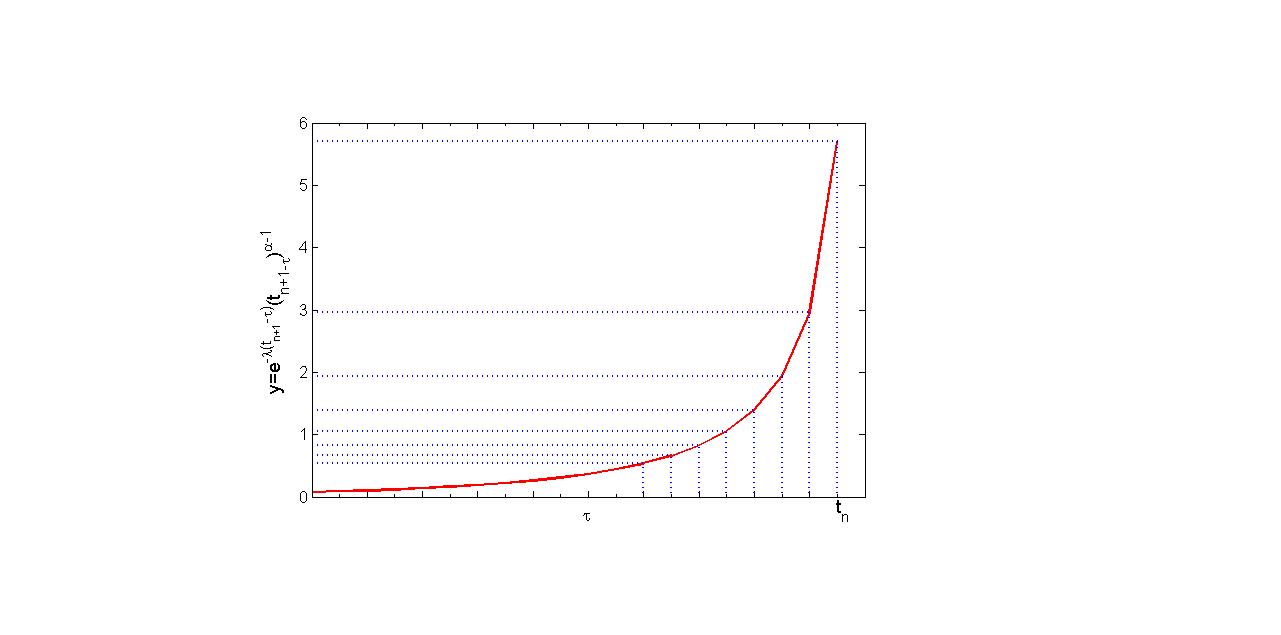}\\
\caption{The relationship between the selected nodes $\{\tau_i\}$ and the function $y=e^{-\lambda(t_{n+1}-\tau)}(t_{n+1}-\tau)^{\alpha-1}$ in the algorithm of equal-area distribution, where $h=1/10$, $\Delta s=h$, $t_{n+1}=2$, $\lambda=1$, and $\alpha=0.2$.}\label{fig:2.3}
\end{figure}
\begin{figure}[!htbp]
\includegraphics[scale=0.4]{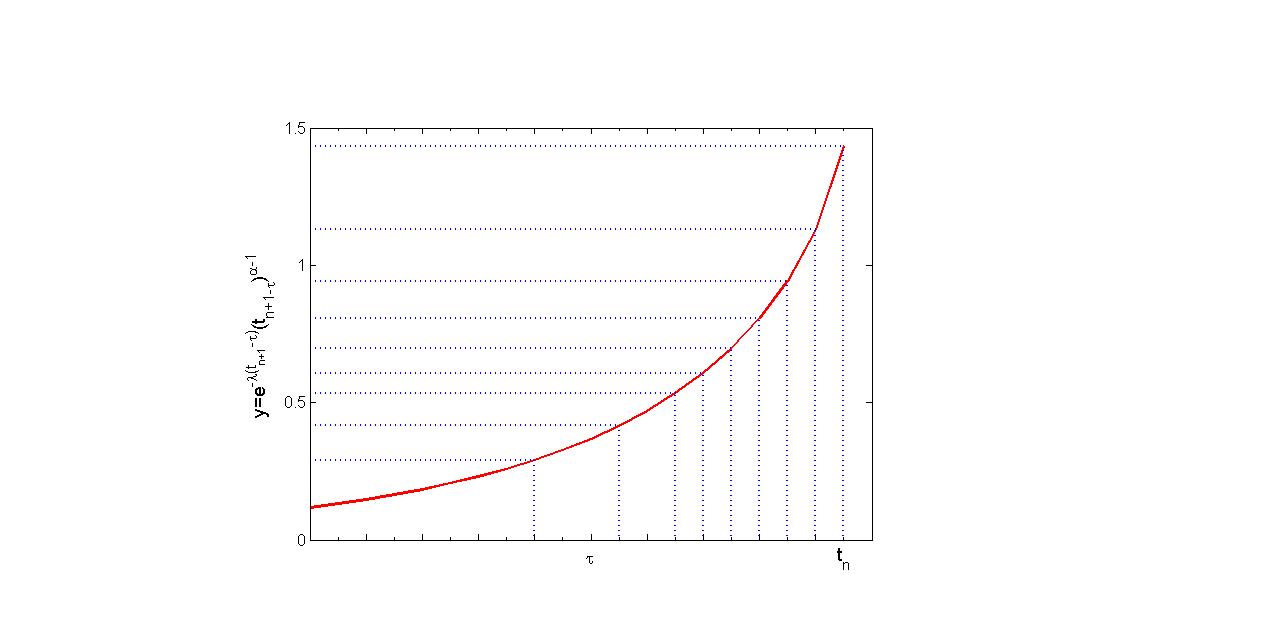}\\
\caption{The relationship between the selected nodes $\{\tau_i\}$ and the function $y=e^{-\lambda(t_{n+1}-\tau)}(t_{n+1}-\tau)^{\alpha-1}$ in the algorithm of equal-area distribution, where $h=1/10$, $\Delta s=h$, $t_{n+1}=2$, $\lambda=1$, and $\alpha=0.8$.}\label{fig:2.4}
\end{figure}
\begin{algorithm}[h]
\caption{The pseudocode of equal-area distribution algorithm for (\ref{e2dot10})-(\ref{e2dot11})}\label{alg:2.2}
\begin{algorithmic}[l]
\STATE $i=0$; $\tau_{i,n}=0$; \% $\tau_{0,n}=0;$\\
\STATE $\tau_c=0$; \% \emph{current node is} $\tau_{0,n}$
\WHILE {$\tau_{c}<t_n$}
    \STATE $\tau_{i+1,n}=\tau_{c}+\Delta s(t_{n+1}-\tau_{c})^{1-\alpha}e^{\lambda(t_{n+1}-\tau_{c})}$;
    \IF {$\tau_{i+1,n}>t_n$}
         \STATE $\tau_{i+1,n}=t_n$; 
         \STATE break;
    \ENDIF\\
    \% \emph{if} $y=e^{-\lambda(t_{n+1}-\tau)}(t_{n+1}-\tau)^{\alpha-1}$
                \emph{changes too fast}
    \IF{$\tau_{i+1,n}-\tau_{c}<h$}
        \STATE  $\tau_{i+1,n}=\tau_{c}+h$;
        \% \emph{make} $\tau_{i+1,n}$ \emph{be one step away from} $\tau_{i,n}$
    \ELSE
        \STATE  $\tau_{i+1,n}=\lfloor \tau_{i+1,n}/h \rfloor\ast h$;
                \% \emph{let} $\tau_{i+1,n}$ \emph{belong to the uniform mesh points} $\{t_j\}_{j=0}^n$
    \ENDIF
    \STATE  $\tau_{c}=\tau_{i+1,n}$; $i=i+1$;
\ENDWHILE
\end{algorithmic}
\end{algorithm}
\begin{remark}\label{remark:2.1}
From Figs. \ref{fig:2.1}-\ref{fig:2.4}, we can see that since the function $y(\tau)=e^{-\lambda(t_{n+1}-\tau)}(t_{n+1}-\tau)^{\alpha-1}$ steepens with the decrease of $\alpha$, for fixed $h$ and fixed $\Delta y$ in the equal-height distribution method or fixed $\Delta s$ in the equal-area one, the less $\alpha$ is, the fewer the quadrature nodes $\{\tau_{i,n}\}$ are selected. This is different from the predictor-corrector approach proposed in \cite{Diethelm:02a,Diethelm:02b,Diethelm:04} or the improved version in \cite{Deng:07d}, in which the integral $\int_{0}^{t_n}\cdot d\tau$ is always approximated on all the mesh points $\{t_j\}_{j=0}^n$ and the nodes number is unrelated to the value of $\alpha$.
\end{remark}

\subsubsection{Predictor-corrector approach for (\ref{e2dot2}) when $\alpha \in (0,1]$}

Since the mesh nodes $\{\tau_{i,n}\}$ chosen from Algorithm \ref{alg:2.1} or Algorithm \ref{alg:2.2} still belong to the set of the uniform nodes, we denote
\begin{equation}\label{e2dot13}
\tau_{i,n}=t_{n_i},~~~i=0,1,\cdots,m_n,
\end{equation}
where $t_j=jh$. Note that $t_{n_0}=0$ and $t_{n_{m_n}}=t_n$. In this way, by using the product trapezoidal quadrature formula to replace the integral, there is
\begin{eqnarray}\label{e2dot14}
&&\int_0^{t_n}e^{-\lambda(t_{n+1}-\tau)}(t_{n+1}-\tau)^{\alpha-1}g(\tau) d \tau
\nonumber\\
&=&\sum_{i=0}^{m_n-1}\int_{\tau_{i,n}}^{\tau_{i+1,n}}
e^{-\lambda(t_{n+1}-\tau)}(t_{n+1}-\tau)^{\alpha-1}g(\tau) d\tau
\nonumber\\
&\approx&\sum_{i=0}^{m_n-1}\int_{\tau_{i,n}}^{\tau_{i+1,n}} (t_{n+1}-\tau)^{\alpha-1}
\frac{e^{-\lambda(t_{n+1}-\tau_{i+1,n})}g(\tau_{i+1,n})(\tau-\tau_{i,n})
+e^{-\lambda(t_{n+1}-\tau_{i,n})}g(\tau_{i,n})(\tau_{i+1,n}-\tau)}
{\tau_{i+1,n}-\tau_{i,n}}
d\tau
\nonumber\\
&=&\sum_{i=0}^{m_n-1}\int_{t_{n_i}}^{t_{n_{i+1}}} (t_{n+1}-\tau)^{\alpha-1}
\frac{e^{-\lambda(t_{n+1}-t_{n_{i+1}})}g(t_{n_{i+1}})(\tau-t_{n_{i}})
+e^{-\lambda(t_{n+1}-t_{n_i})}g(t_{n_{i}})(t_{n_{i+1}}-\tau)}
{t_{n_{i+1}}-t_{n_{i}}}
d\tau
\nonumber\\
&=&\sum_{i=0}^{m_n-1}\frac{e^{-\lambda(t_{n+1}-t_{n_{i+1}})}g(t_{n_{i+1}})}{(n_{i+1}-n_{i})h}
\int_{t_{n_i}}^{t_{n_{i+1}}}(t_{n+1}-\tau)^{\alpha-1}(\tau-t_{n_{i}}) d\tau
\nonumber\\
&&+\sum_{i=0}^{m_n-1}\frac{e^{-\lambda(t_{n+1}-t_{n_{i}})}g(t_{n_{i}})}{(n_{i+1}-n_{i})h}
\int_{t_{n_i}}^{t_{n_{i}}}(t_{n+1}-\tau)^{\alpha-1}(t_{n_{i+1}}-\tau) d\tau
\nonumber\\
&=&\sum_{i=1}^{m_n}\frac{h^\alpha e^{-\lambda(t_{n+1}-t_{n_{i}})}g(t_{n_{i}})}{(n_{i}-n_{i-1})}
\left[-\frac{(n_{i}-n_{i-1})(n+1-n_i)^\alpha}{\alpha}
-\frac{(n+1-n_{i})^{\alpha+1}-(n+1-n_{i-1})^{\alpha+1}}{\alpha(\alpha+1)}\right]
\nonumber\\
&&+\sum_{i=0}^{m_n-1}\frac{h^\alpha e^{-\lambda(t_{n+1}-t_{n_{i}})}g(t_{n_{i}})}{(n_{i+1}-n_{i})}
\left[\frac{(n_{i+1}-n_{i})(n+1-n_i)^\alpha}{\alpha}
+\frac{(n+1-n_{i+1})^{\alpha+1}-(n+1-n_{i})^{\alpha+1}}{\alpha(\alpha+1)}\right]
\nonumber\\
&=&\sum_{i=1}^{m_n-1}\frac{h^\alpha e^{-\lambda(t_{n+1}-t_{n_{i}})}g(t_{n_{i}})}{\alpha(\alpha+1)}\cdot
\nonumber\\
&&~~~~~~\left[\frac{(n+1-n_{i+1})^{\alpha+1}-(n+1-n_{i})^{\alpha+1}}{n_{i+1}-n_{i}}
-\frac{(n+1-n_{i})^{\alpha+1}-(n+1-n_{i-1})^{\alpha+1}}{n_{i}-n_{i-1}}\right]
\nonumber\\
&&+\frac{h^\alpha e^{-\lambda t_{n+1}}g(t_{0})}{n_1\alpha(\alpha+1)}
\left\{(n+1-n_1)^{\alpha+1}-(n+1)^\alpha[n+1-(\alpha+1)n_1]\right\}
\nonumber\\
&&+\frac{h^\alpha g(t_{n})e^{-\lambda h}}{(n-n_{m_n-1})\alpha(\alpha+1)}
\left\{(n+1-n_{m_n-1})^{\alpha+1}-(\alpha+1)(n-n_{m_n-1})-1\right\}.
\end{eqnarray}
Assume that the approximations $x_j\approx x(t_j)$, $j=1,2,\cdots,n$, have been obtained. Then by combining with the one step approximations (\ref{e2dot4}) and (\ref{e2dot5}), the predictor-corrector approach to compute $x_{n+1}\approx x(t_{n+1})$ for $\alpha\in(0,1]$ can be yielded as

{\bf Case 1 ($n=0$):}
\begin{equation}\label{e2dot15}
\left\{
\begin{array}{l}
\displaystyle x_1^{Pr}=x_0(t_1)+\frac{h^\alpha}{\Gamma(\alpha+1)} e^{-\lambda h}f(0,x_0), \\
 \\
\displaystyle x_1=x_0(t_1)+\frac{h^\alpha}{\Gamma(\alpha+2)}\left[f(h,x_1^{Pr}) +\alpha e^{-\lambda h}f(0,x_0)\right];
\end{array}
\right.
\end{equation}

{\bf Case 2 ($n\geq 1$):}
\begin{equation}\label{e2dot16}
\left\{
\begin{array}{ll}
 x_{n+1}^{Pr} & =  \displaystyle x_0(t_{n+1})+
\frac{h^\alpha}{\Gamma(\alpha+2)}
\left[\sum\limits_{i=0}^{m_n-1} a_{i,n+1}f(t_i,x_{n_i})
+b_{n} f(t_n,x_n)\right], \\
\\
x_{n+1} & =  \displaystyle x_0(t_{n+1})+
\frac{h^\alpha}{\Gamma(\alpha+2)}
\left[\sum\limits_{i=0}^{m_n} a_{i,n+1}f(t_i,x_{n_i})
+f(t_{n+1},x_{n+1}^{Pr})\right],
\end{array}
\right.
\end{equation}
where
\begin{equation}\label{e2dot17}
a_{i,n+1}= \left\{
\begin{array}{ll}
\frac{e^{-\lambda (n+1)h}}{n_1}
\left\{(n+1-n_1)^{\alpha+1}-(n+1)^\alpha[n+1-(\alpha+1)n_1]\right\},
& i=0,\\
\\
e^{-\lambda (n+1-n_i) h}\cdot
\left[\frac{(n+1-n_{i+1})^{\alpha+1}-(n+1-n_{i})^{\alpha+1}}{n_{i+1}-n_{i}}
-\frac{(n+1-n_{i})^{\alpha+1}-(n+1-n_{i-1})^{\alpha+1}}{n_{i}-n_{i-1}}\right],
& 1\leq i\leq m_n-1,\\
\\
\frac{e^{-\lambda h}}{n-n_{m_n-1}}\left[(n+1-n_{m_n-1})^{\alpha+1}-(n+1-n_{m_n-1})\right],
&i=m_n;
\end{array}
 \right.
\end{equation}
and
\begin{equation}\label{e2dot18}
b_{n}=a_{m_n,n+1}+e^{-\lambda h}.
\end{equation}

If no equidistributing principle is adopted, i.e., $n_i=i$, the predictor-corrector approach for (\ref{e2dot2}) to compute $x_{n+1}\approx x(t_{n+1})$ is described as

{\bf Case 1 ($n=0$):} being the same as (\ref{e2dot15});

{\bf Case 2 ($n\geq 1$):}
\begin{equation}\label{e2dot199}
\left\{
\begin{array}{ll}
 x_{n+1}^{Pr} & =  \displaystyle x_0(t_{n+1})+
\frac{h^\alpha}{\Gamma(\alpha+2)}
\left[\sum\limits_{i=0}^{n-1} d_{i,n+1}f(t_i,x_i)+e^{-\lambda h}(2^{\alpha+1}-1)\cdot f(t_n,x_n)
\right], \\
\\
x_{n+1} & =  \displaystyle x_0(t_{n+1})+
\frac{h^\alpha}{\Gamma(\alpha+2)}
\left[\sum\limits_{i=0}^{n} d_{i,n+1}f(t_i,x_i)
+f(t_{n+1},x_{n+1}^{Pr})\right],
\end{array}
\right.
\end{equation}
where
\begin{equation}\label{e2dot200}
d_{i,n+1}= \left\{
\begin{array}{ll}
e^{-\lambda (n+1)h}
\left(n^{\alpha+1}-(n+1)^\alpha(n-\alpha)\right),
& i=0,\\
\\
e^{-\lambda (n+1-i) h}\cdot
\left((n-i)^{\alpha+1}-2(n+1-i)^{\alpha+1}
+(n+2-i)^{\alpha+1}\right),
& 1\leq i\leq n.
\end{array}
 \right.
\end{equation}

\begin{theorem}
 For the tempered fractional initial value problem (\ref{e1dot1})-(\ref{e1dot2}),
assume $f(t,x(t))$ of (\ref{e1dot1}) belongs to  $C^2[0,T]$ for some suitable $T$. Then for the scheme (\ref{e2dot199})-(\ref{e2dot200}) we have
$$
\max\limits_{0 \le n \le N} |x(t_n)-x_n|=\left\{ \begin{array}{ll}
O(h^2), & \textrm{if} ~~ \alpha \ge 0.5, \\
O(h^{1+2 \alpha}), & \textrm{if} ~~ 0<\alpha < 0.5.
\end{array}
\right.
$$
\end{theorem}
\begin{proof}
Combining Appendix A of \cite{Deng:07d} with Theorem 2.5 and Lemma 3.1 of \cite{Diethelm:04}, we can obtain the result of this theorem.
\end{proof}

\subsection{Algorithms for (\ref{e2dot3}) when $\alpha \in (0,2)$}

Now we apply the predictor-corrector algorithms to Eq. (\ref{e2dot3}) for $\alpha \in(0,2)$. To begin, we analyze the properties of the kernel function $y(\tau)=e^{-\lambda (t_{n+1}-\tau)}(t_{n+1}-\tau)^{\alpha-1}
-e^{-\lambda (t_{n}-\tau)}(t_n-\tau)^{\alpha-1}$. 
On one hand, by noticing $\lambda>0$, there is
\begin{eqnarray}
y(\tau)&=&e^{-\lambda (t_{n+1}-\tau)}(t_{n+1}-\tau)^{\alpha-1}
-e^{-\lambda (t_{n}-\tau)}(t_n-\tau)^{\alpha-1}
\nonumber\\
&\leq& e^{-\lambda (t_{n+1}-\tau)}
\left[(t_{n+1}-\tau)^{\alpha-1}-(t_n-\tau)^{\alpha-1}\right];
\end{eqnarray} together with the short memory principle
\cite{Ford:01,Deng:07c} that 
$(t_{n+1}-\tau)^{\alpha-1}-(t_n-\tau)^{\alpha-1}$ decays with the
order $2-\alpha$, we can see that $y(\tau)$ decays faster than exponential order $\lambda$ for fixed $\tau$ when $t_{n}\rightarrow \infty$. On the other hand, it can be seen from
\begin{equation}
\left[e^{-\lambda x}x^{\alpha-1}\right]'=e^{-\lambda x}x^{\alpha-2}\left[(\alpha-1)-\lambda x\right]
\end{equation}
that $y(\tau)<0$ if $\alpha \in(0,1]$; and if $\alpha \in(1,2)$ and when $t_{n+1}-\tau<\frac{\alpha-1}{\lambda}$, i.e., $\tau>t_{n+1}-\frac{\alpha-1}{\lambda}$, there is $y(\tau)>0$; while if $\alpha \in(1,2)$ and when $t_{n}-\tau>\frac{\alpha-1}{\lambda}$, i.e., $\tau<t_{n}-\frac{\alpha-1}{\lambda}$, there is $y(\tau)<0$. Also, $y(\tau)$ is lower bounded. Specifically,
\begin{eqnarray}
y(\tau)&=&e^{-\lambda (t_{n+1}-\tau)}(t_{n+1}-\tau)^{\alpha-1}
-e^{-\lambda (t_{n}-\tau)}(t_n-\tau)^{\alpha-1}
\nonumber\\
&\geq&(t_n-\tau)^{\alpha-1}
\left[e^{-\lambda (t_{n+1}-\tau)}-e^{-\lambda (t_{n}-\tau)}\right]
\nonumber\\
&=&(t_n-\tau)^{\alpha-1}e^{-\lambda (t_{n+1}-\tau)}(1-e^{\lambda h})
\nonumber\\
&=&-(t_n-\tau)^{\alpha-1}e^{-\lambda (t_{n+1}-\tau)}\cdot O(h),
\end{eqnarray}
that is, $y(\tau)$ is very close to zeros when it is negative; thus, we can directly neglect the part when $y(\tau)<0$.

Basing on this fact, similar to the case of dealing with (\ref{e2dot2}) when $\alpha\in (0,1]$, we can explore more efficient ways to approximate the integral $\int_0^{t_n}\left(e^{-\lambda (t_{n+1}-\tau)}(t_{n+1}-\tau)^{\alpha-1}-e^{-\lambda (t_n-\tau)}(t_n-\tau)^{\alpha-1}\right)g(\tau)d \tau$ with trapezoidal quadrature formula. Corresponding to Algorithm \ref{alg:2.1} and Algorithm \ref{alg:2.2}, we provide equal-height as well as equal-area distribution methods for choosing the quadrature nodes in the subsection below.

\subsubsection{Two equidistributing options for selecting the quadrature nodes}

Suppose that those to be chosen points $\{\tau_{i,n}\}$ are sequential as in (\ref{e2dot6}), and we have already got the points $\tau_{i,n}$, $0\leq i<m_n$. By the analysis above, we further let $\tau_{1,n}\geq \min\{t_n,t_{n+1}-\frac{\alpha-1}{\lambda}\}$ for $\alpha \in (1,2)$. Like the equal-height distribution principle in Subsection $2.1$, an intuitive way of selecting the next point $\tau_{i+1,n}$ is, firstly, let
\begin{eqnarray}\label{e2dot19}
&&e^{-\lambda (t_{n+1}-\tilde{\tau}_{i+1,n})}(t_{n+1}-\tilde{\tau}_{i+1,n})^{\alpha-1}
-e^{-\lambda (t_{n}-\tilde{\tau}_{i+1,n})}(t_n-\tilde{\tau}_{i+1,n})^{\alpha-1}
\nonumber\\
&=&e^{-\lambda (t_{n+1}-\tau_{i})^{\alpha-1}}(t_{n+1}-\tau_{i})^{\alpha-1}
-e^{-\lambda (t_{n}-\tau_{i})^{\alpha-1}}(t_n-\tau_{i,n})^{\alpha-1}+(-1)^{\lceil \alpha \rceil}\Delta y,
\end{eqnarray}
where $\Delta y$ is a given small positive real number, and if $\alpha\in(0,1)$, $-\Delta y$ is used, else, if $\alpha \in (1,2)$, it is replaced by $\Delta y$. That is,
\begin{equation}\label{e2dot20}
\begin{array}{l}
e^{-\lambda \left[t_{n+1}-\tau_{i,n}-(\tilde{\tau}_{i+1,n}-\tau_{i,n})\right]}
\left[t_{n+1}-\tau_{i,n}-(\tilde{\tau}_{i+1,n}-\tau_{i,n})\right]^{\alpha-1}
-e^{-\lambda (t_{n+1}-\tau_{i})}(t_{n+1}-\tau_{i})^{\alpha-1}
\\
=e^{-\lambda \left[t_n-\tau_{i,n}-(\tilde{\tau}_{i+1,n}-\tau_{i,n})\right]}
\left[t_n-\tau_{i,n}-(\tilde{\tau}_{i+1,n}-\tau_{i,n})\right]^{\alpha-1}
-e^{-\lambda (t_n-\tau_{i,n})}(t_n-\tau_{i,n})^{\alpha-1}+(-1)^{\lceil \alpha \rceil}\Delta y.
\end{array}
\end{equation}
Since the above equation is nonlinear w.r.t. $\tilde{\tau}_{i+1,n}$, we use its linear part instead to select the wanted point. That is,
\begin{equation}\label{e2dot21}
\tilde{\tau}_{i+1,n} =\max\left\{\begin{array}{l}
\textrm{solve}\left(\tilde{\tau}_{i+1,n}-\tau_{i,n}=h, \tilde{\tau}_{i+1,n}  \right),
\\
\textrm{solve}\big((\tilde{\tau}_{i+1,n}-\tau_{i,n})\left[\lambda(t_{n+1}-\tau_{i,n})^{\alpha-1}
-(\alpha-1)(t_{n+1}-\tau_{i,n})^{\alpha-2}\right]
\\
~~~~~~~=(\tilde{\tau}_{i+1,n}-\tau_{i,n})e^{\lambda h}\left[\lambda(t_{n}-\tau_{i,n})^{\alpha-1}
-(\alpha-1)(t_{n}-\tau_{i,n})^{\alpha-2}\right]
\\
~~~~~~~~~~+(-1)^{\lceil \alpha \rceil}\Delta y e^{\lambda (t_{n+1}-\tau_{i,n})},  \tilde{\tau}_{i+1,n}  \big)
\end{array}
\right\}.
\end{equation}
Then, to avoid being non-equal divided nodes, we take
\begin{equation}\label{e2dot22}
\tau_{i+1,n}=\left\lfloor\frac{\tilde{\tau}_{i+1,n}}{h}\right\rfloor\cdot h.
\end{equation}
The above descriptions are demonstrated more clearly by Figs. \ref{fig:2.5_add}-\ref{fig:2.6} and Algorithm \ref{alg:2.3}.
\begin{figure}[!htbp]
\includegraphics[scale=0.4]{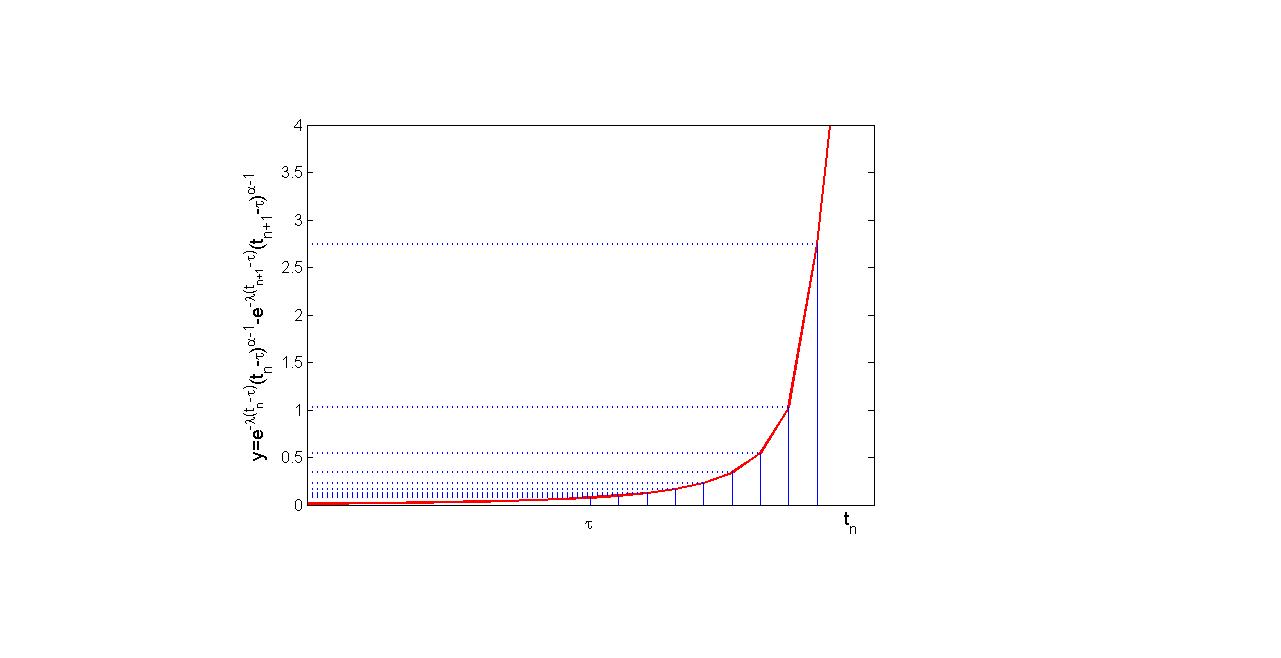}\\
\caption{The relationship between the selected nodes $\{\tau_i\}$ and the function $y=e^{-\lambda (t_{n}-\tau)}(t_{n}-\tau)^{\alpha-1}-
e^{-\lambda (t_{n+1}-\tau)}(t_{n+1}-\tau)^{\alpha-1}$ in the algorithm of equal-height distribution, where $h=1/10$, $\Delta y=h/5$, $t_{n+1}=2$, $\lambda=1$, and $\alpha=0.2$.}\label{fig:2.5_add}
\end{figure}
\begin{figure}[!htbp]
\includegraphics[scale=0.4]{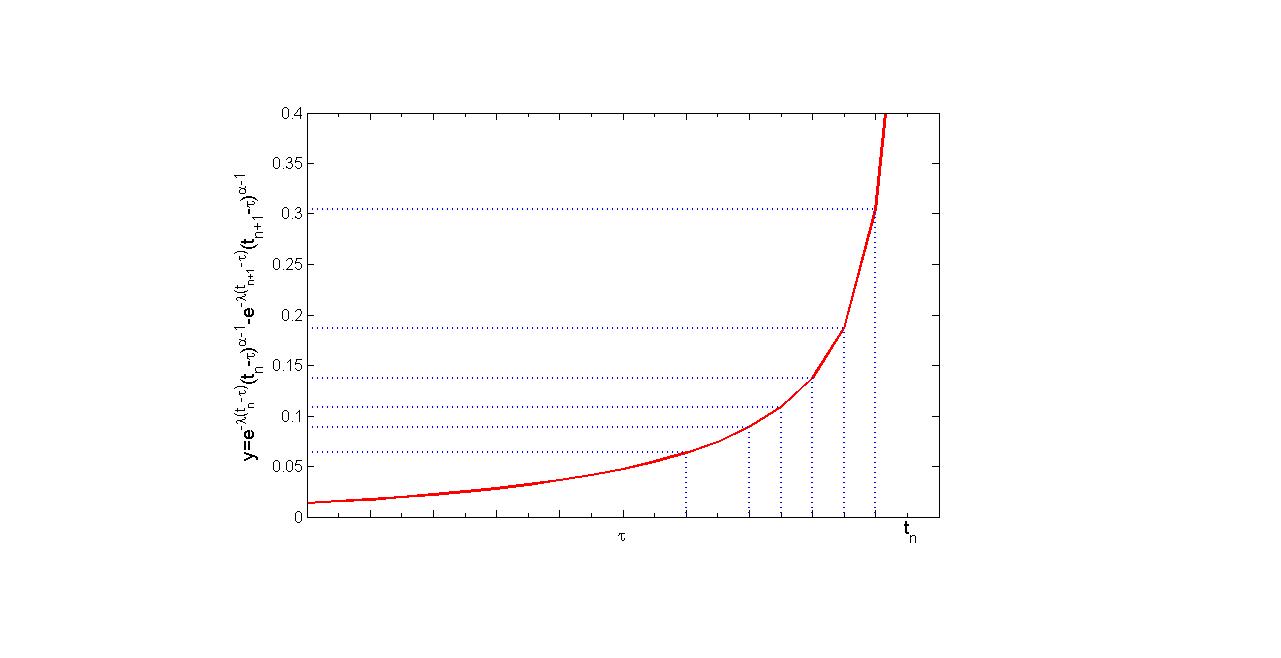}\\
\caption{The relationship between the selected nodes $\{\tau_i\}$ and the function $y=e^{-\lambda (t_{n}-\tau)}(t_{n}-\tau)^{\alpha-1}-
e^{-\lambda (t_{n+1}-\tau)}(t_{n+1}-\tau)^{\alpha-1}$ in the algorithm of equal-height distribution, where $h=1/10$, $\Delta y=h/5$, $t_{n+1}=2$, $\lambda=1$, and $\alpha=0.8$.}\label{fig:2.6_add}
\end{figure}
\begin{figure}[!htbp]
\includegraphics[scale=0.4]{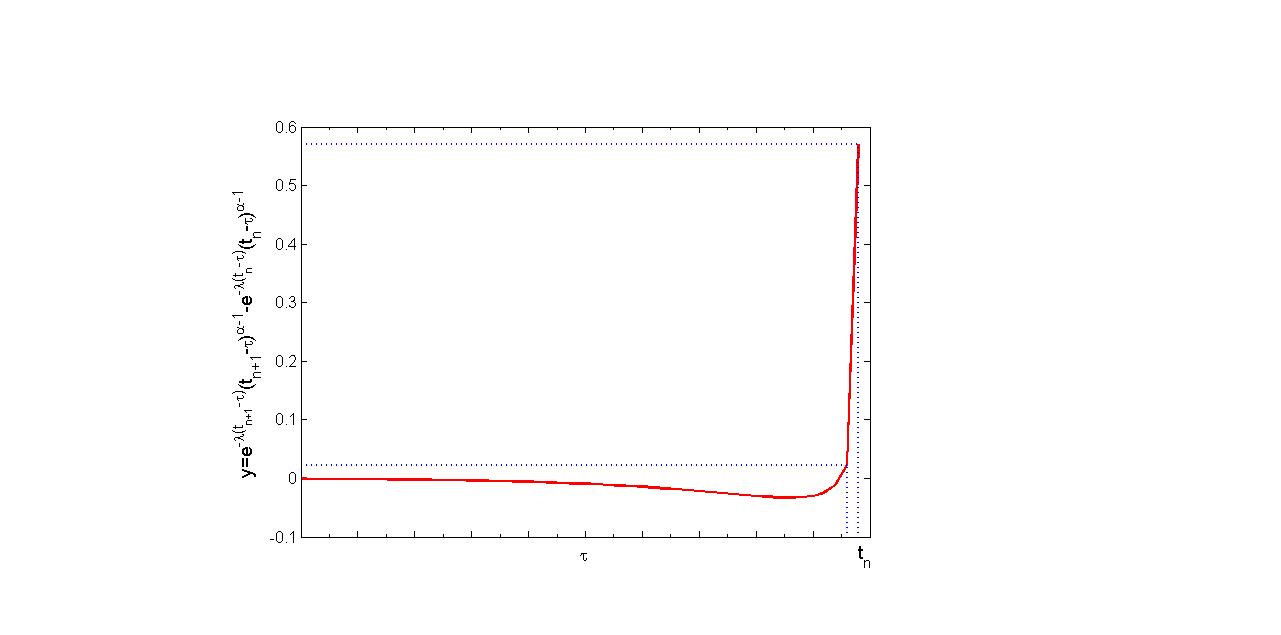}\\
\caption{The relationship between the selected nodes $\{\tau_i\}$ and the function $y=e^{-\lambda (t_{n+1}-\tau)}(t_{n+1}-\tau)^{\alpha-1}-
e^{-\lambda (t_{n}-\tau)}(t_{n}-\tau)^{\alpha-1}$ in the algorithm of equal-height distribution, where $h=1/10$, $\Delta y=h/5$, $t_{n+1}=5$, $\lambda=1$, and $\alpha=1.2$.}\label{fig:2.5}
\end{figure}
\begin{figure}[!htbp]
\includegraphics[scale=0.4]{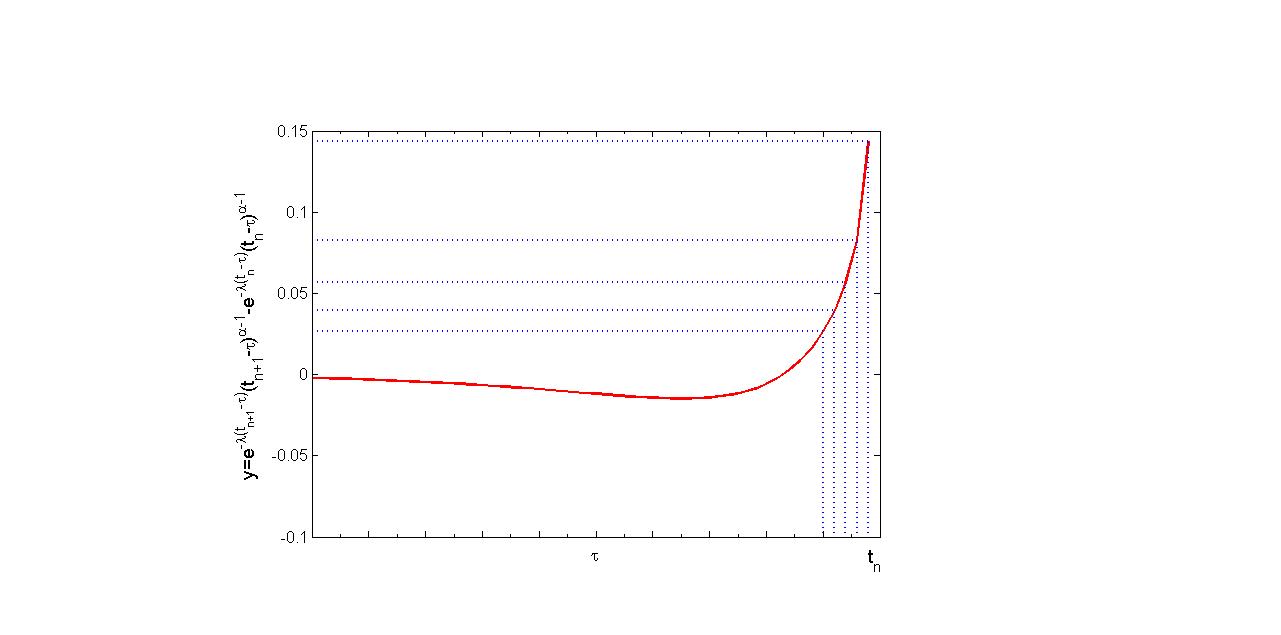}\\
\caption{The relationship between the selected nodes $\{\tau_i\}$ and the function $y=e^{-\lambda (t_{n+1}-\tau)}(t_{n+1}-\tau)^{\alpha-1}-
e^{-\lambda (t_{n}-\tau)}(t_{n}-\tau)^{\alpha-1}$ in the algorithm of equal-height distribution, where $h=1/10$, $\Delta y=h/5$, $t_{n+1}=5$, $\lambda=1$, and $\alpha=1.8$.}\label{fig:2.6}
\end{figure}
\begin{algorithm}[h]
\caption{The pseudocode of equal-height distribution algorithm for (\ref{e2dot21})-(\ref{e2dot22})}\label{alg:2.3}
\begin{algorithmic}[l]
\STATE $i=0$; $\tau_{i,n}=0$; \% $\tau_{0,n}=0;$\\
\STATE $\tau_c=0$; \% \emph{current node is} $\tau_{0,n}$
\IF {$0<\alpha\leq 1$}
    \STATE $\Delta y=-\Delta y$; \% \emph{if $\alpha \in (0,1]$, take negative value}
\ELSE
    \STATE $\tau_{1,n}=t_{n+1}-\frac{\alpha-1}{\lambda}$;
            \% \emph{start from the point} $t_{n+1}-\frac{\alpha-1}{\lambda}$, \emph{where} $y(\tau)>0$
    \STATE $\tau_{1,n}=\lfloor \tau_{i,n}/h\rfloor \cdot h$;
            \% \emph{let} $\tau_{1,n}$ \emph{belong to the uniform mesh} $\{t_j\}_{j=0}^n$
    \STATE $i=1$; $\tau_c=\tau_{1,n}$; \% \emph{current node is} $\tau_{1,n}$
\ENDIF
\WHILE {$\tau_{c}<t_n$}
    \STATE $\tau_{i+1,n}=\tau_{c}+\frac{\Delta y e^{\lambda (t_{n+1}-\tau_{c})}}
            {\left[\lambda(t_{n+1}-\tau_{c})^{\alpha-1}
            -(\alpha-1)(t_{n+1}-\tau_{c})^{\alpha-2}\right]-e^{\lambda h}
            \left[\lambda(t_{n}-\tau_{c})^{\alpha-1}
            -(\alpha-1)(t_{n}-\tau_{c})^{\alpha-2}\right]}$;
    \IF {$\tau_{i+1,n}>t_n$}
         \STATE $\tau_{i+1,n}=t_n$; 
         \STATE break;
    \ENDIF\\
    \% \emph{if} $y=e^{-\lambda(t_{n+1}-\tau)}(t_{n+1}-\tau)^{\alpha-1}-e^{-\lambda(t_n-\tau)}(t_{n}-\tau)^{\alpha-1}$
                \emph{changes too fast}
    \IF{$\tau_{i+1,n}-\tau_{c}<h$}
        \STATE  $\tau_{i+1,n}=\tau_{c}+h$;
        \% \emph{make} $\tau_{i+1,n}$ \emph{be one step away from} $\tau_{i,n}$
    \ELSE
        \STATE  $\tau_{i+1,n}=\lfloor\tau_{i+1,n}/h\rfloor\ast h$;
                \% \emph{let} $\tau_{i+1,n}$ \emph{belong to the uniform mesh points} $\{t_j\}_{j=0}^n$
    \ENDIF
    \STATE  $\tau_{c}=\tau_{i+1,n}$; $i=i+1$;
\ENDWHILE
\end{algorithmic}
\end{algorithm}

Similar to the equal-area idea in Subsection $2.1$, here for the second way of choosing the mesh points $\{\tau_{i,n}\}$, we expect firstly that
\begin{equation}\label{e2dot23}
\int_{\tau_i}^{\tilde{\tau}_{i+1}}
 \left[ e^{-\lambda (t_{n+1}-\tau)}(t_{n+1}-\tau)^{\alpha-1}
-e^{-\lambda (t_{n}-\tau)}(t_{n}-\tau)^{\alpha-1}\right]d\tau
=(-1)^{\lceil \alpha \rceil}\Delta s,
\end{equation}
where $\Delta s$ is a given small positive real number, and if $\alpha \in (0,1)$, $-\Delta s$ is used, else, if $\alpha \in (1,2)$, it is replaced by $\Delta s$.
Again, since it is a nonlinear equations w.r.t. $\tilde{\tau}_{i+1,n}$, we approximate it by
\begin{equation}
(t_{n+1}-\tau_{i,n})^{\alpha-1}\int_{\tau_i}^{\tilde{\tau}_{i+1}}
\left[ e^{-\lambda (t_{n+1}-\tau)}-e^{-\lambda (t_{n}-\tau)}\right]d\tau
=(-1)^{\lceil \alpha \rceil}\Delta s.
\end{equation}
So, let
\begin{equation}\label{e2dot25}
\tilde{\tau}_{i+1,n}=
 \max\left\{
\begin{array}{l}
\textrm{solve}\left(\tilde{\tau}_{i+1,n}-\tau_{i,n}=h, \tilde{\tau}_{i+1,n}\right),
\\
\textrm{solve}\big(\lambda(\tilde{\tau}_{i+1,n}-\tau_{i,n})
\left[(t_{n+1}-\tau_{i,n})^{\alpha-1}
-e^{\lambda h}(t_{n}-\tau_{i,n})^{\alpha-1}\right]
\\~~~~~~~=(-1)^{\lceil \alpha \rceil}\Delta s \lambda e^{\lambda (t_{n+1}-\tau_{i,n})}, \tilde{\tau}_{i+1,n}\big)
\end{array}
\right\}.
\end{equation}
Then by taking
\begin{equation}\label{e2dot26}
\tau_{i+1,n}=\left\lfloor\frac{\tilde{\tau}_{i+1,n}}{h}\right\rfloor\cdot h,
\end{equation}
$\tau_{i+1,n}$ belongs to uniform grid points $\{t_j\}_{j=0}^{n}$.
We can have a better view of this equal-area distribution criterion of selecting the mesh $\{\tau_{i,n}\}$ from Figs. \ref{fig:2.7_add}-\ref{fig:2.8} and Algorithm \ref{alg:2.4}.
\begin{figure}[!htbp]
\includegraphics[scale=0.4]{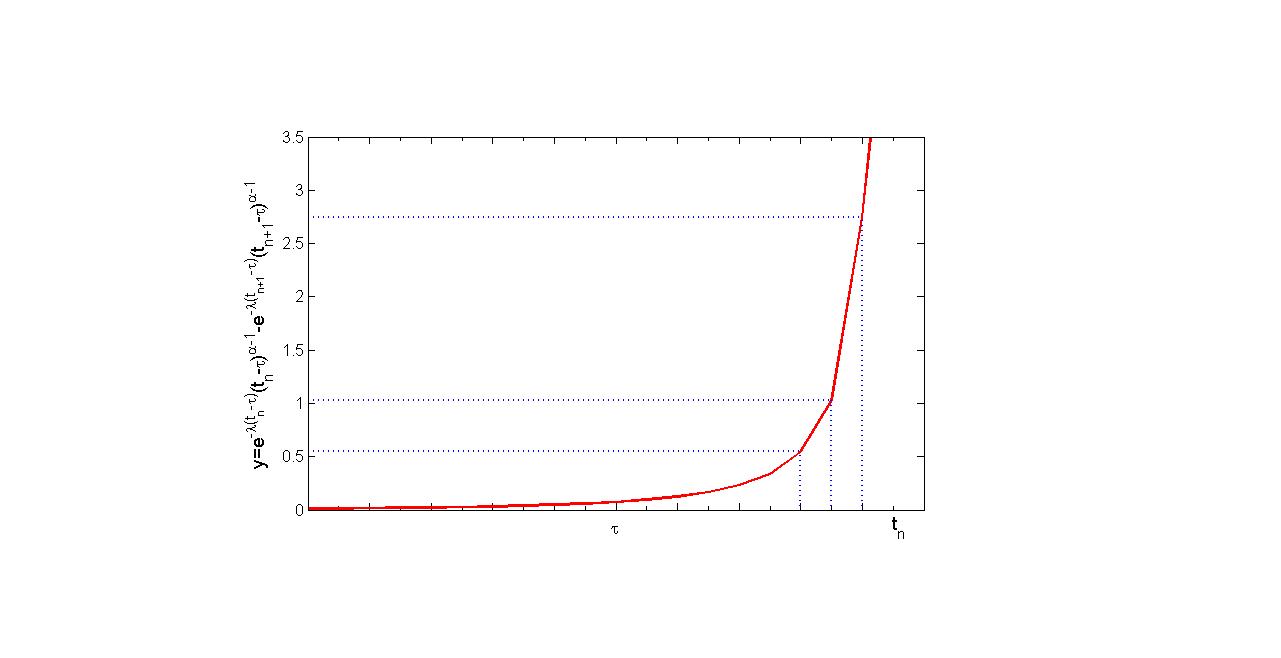}\\
\caption{The relationship between the selected nodes $\{\tau_i\}$ and the function $y=e^{-\lambda (t_{n}-\tau)}(t_{n}-\tau)^{\alpha-1}
-e^{-\lambda (t_{n+1}-\tau)}(t_{n+1}-\tau)^{\alpha-1}$ in the algorithm of equal-area distribution, where $h=1/10$, $\Delta s=h/5$, $t_{n+1}=2$, $\lambda=1$, and $\alpha=0.2$.}\label{fig:2.7_add}
\end{figure}
\begin{figure}[!htbp]
\includegraphics[scale=0.4]{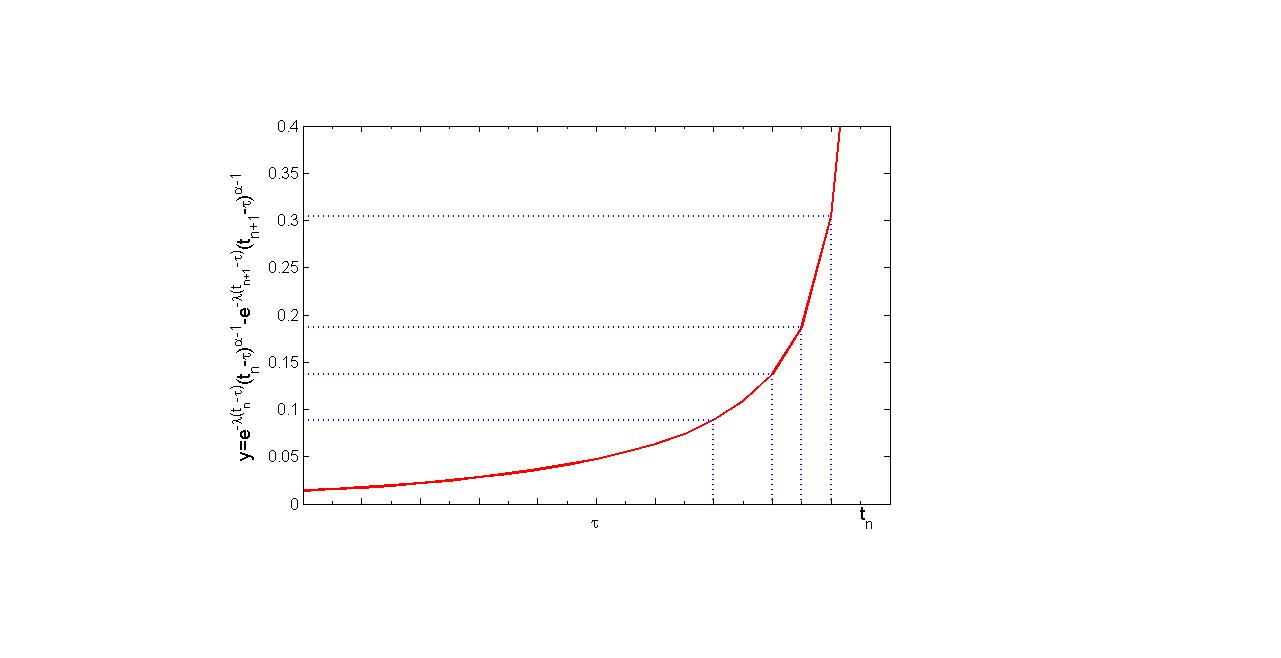}\\
\caption{The relationship between the selected nodes $\{\tau_i\}$ and the function $y=e^{-\lambda (t_{n}-\tau)}(t_{n}-\tau)^{\alpha-1}
-e^{-\lambda (t_{n+1}-\tau)}(t_{n+1}-\tau)^{\alpha-1}$ in the algorithm of equal-area distribution, where $h=1/10$, $\Delta s=h/5$, $t_{n+1}=2$, $\lambda=1$, and $\alpha=0.8$.}\label{fig:2.8_add}
\end{figure}
\begin{figure}[!htbp]
\includegraphics[scale=0.4]{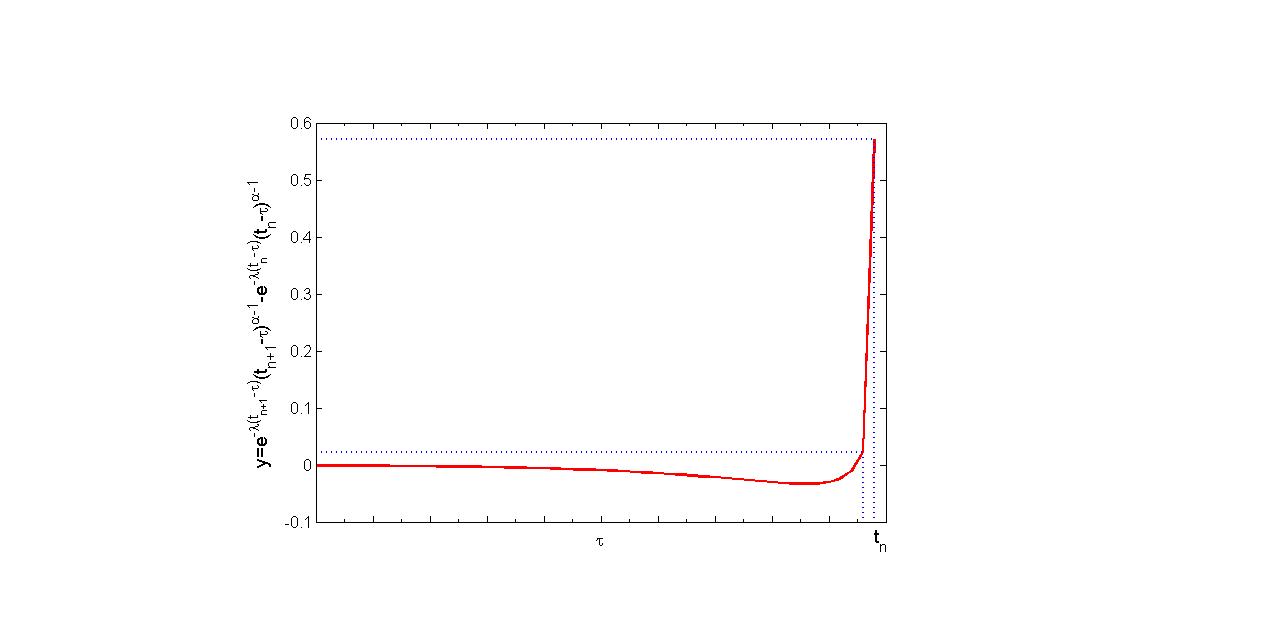}\\
\caption{The relationship between the selected nodes $\{\tau_i\}$ and the function $y=e^{-\lambda (t_{n+1}-\tau)}(t_{n+1}-\tau)^{\alpha-1}
-e^{-\lambda (t_{n}-\tau)}(t_{n}-\tau)^{\alpha-1}$ in the algorithm of equal-area distribution, where $h=1/10$, $\Delta s=h/5$, $t_{n+1}=5$, $\lambda=1$, and $\alpha=1.2$.}\label{fig:2.7}
\end{figure}
\begin{figure}[!htbp]
\includegraphics[scale=0.4]{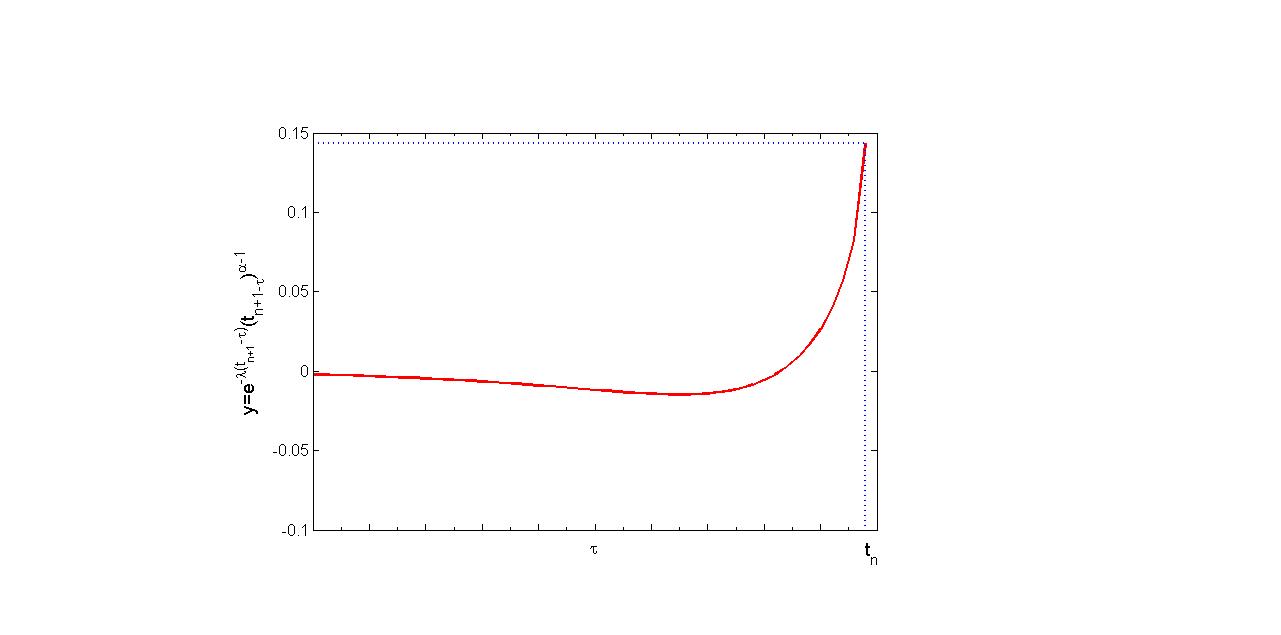}\\
\caption{The relationship between the selected nodes $\{\tau_i\}$ and the function $y=e^{-\lambda (t_{n+1}-\tau)}(t_{n+1}-\tau)^{\alpha-1}
-e^{-\lambda (t_{n}-\tau)}(t_{n}-\tau)^{\alpha-1}$ in the algorithm of equal-area distribution, where $h=1/10$, $\Delta s=h/5$, $t_{n+1}=5$, $\lambda=1$, and $\alpha=1.8$.}\label{fig:2.8}
\end{figure}
\begin{algorithm}[h]
\caption{The pseudocode of equal-area distribution algorithm for (\ref{e2dot25})-(\ref{e2dot26})}\label{alg:2.4}
\begin{algorithmic}[l]
\STATE $i=0$; $\tau_{i,n}=0$; \% $\tau_{0,n}=0;$\\
\STATE $\tau_c=0$; \% \emph{current node is} $\tau_{0,n}$
\IF {$0<\alpha\leq 1$}
    \STATE $\Delta s=-\Delta s$; \% \emph{if $\alpha \in (0,1]$, take negative value}
\ELSE
    \STATE $\tau_{1,n}=t_{n+1}-\frac{\alpha-1}{\lambda}$;
            \% \emph{start from the point} $t_{n+1}-\frac{\alpha-1}{\lambda}$, \emph{where} $y(\tau)>0$
    \STATE $\tau_{1,n}=\lfloor \tau_{i,n}/h\rfloor \cdot h$;
            \% \emph{let} $\tau_{1,n}$ \emph{belong to the uniform mesh} $\{t_j\}_{j=0}^n$
    \STATE $i=1$; $\tau_c=\tau_{1,n}$; \% \emph{current node is} $\tau_{1,n}$
\ENDIF
\WHILE {$\tau_{c}<t_n$}
    \STATE $\tau_{i+1,n}=\tau_{c}+
            \frac{\Delta s e^{\lambda (t_{n+1}-\tau_{c})}}
            {\left[(t_{n+1}-\tau_{c})^{\alpha-1}
            -e^{\lambda h}(t_{n}-\tau_{c})^{\alpha-1}\right]}$;
    \IF {$\tau_{i+1,n}>t_n$}
         \STATE $\tau_{i+1,n}=t_n$; 
         \STATE break;
    \ENDIF\\
    \% \emph{if} $y=e^{-\lambda(t_{n+1}-\tau)}(t_{n+1}-\tau)^{\alpha-1}-e^{-\lambda(t_n-\tau)}(t_{n}-\tau)^{\alpha-1}$
                \emph{changes too fast}
    \IF{$\tau_{i+1,n}-\tau_{c}<h$}
        \STATE  $\tau_{i+1,n}=\tau_{c}+h$;
        \% \emph{make} $\tau_{i+1,n}$ \emph{be one step away from} $\tau_{i,n}$
    \ELSE
        \STATE  $\tau_{i+1,n}= \lfloor \tau_{i+1,n}/h \rfloor \ast h $;
                \% \emph{let} $\tau_{i+1,n}$ \emph{belong to the uniform mesh points} $\{t_j\}_{j=0}^n$
    \ENDIF
    \STATE  $\tau_{c}=\tau_{i+1,n}$; $i=i+1$;
\ENDWHILE
\end{algorithmic}
\end{algorithm}

\subsubsection{Predictor-corrector approach for (\ref{e2dot3}) when $\alpha \in (0,2)$}

We still denote the mesh nodes $\{\tau_{i,n}\}$ chosen from Algorithm \ref{alg:2.3} or Algorithm \ref{alg:2.4} as done in (\ref{e2dot13}).
Also, by noticing that $t_{n_0}=0$ and $t_{n_{m_n}}=t_n$, from (\ref{e2dot15}), we can get
\begin{eqnarray}\label{e2dot32}
&&\int_0^{t_{n}}\left[e^{-\lambda (t_{n+1}-\tau)}(t_{n+1}-\tau)^{\alpha-1}
-e^{-\lambda (t_{n}-\tau)}(t_{n+1}-\tau)^{\alpha-1}\right]g(\tau)d \tau
\nonumber\\
&=&\sum_{i=0}^{m_n-1}\int_{\tau_{i,n}}^{\tau_{i+1,n}}
e^{-\lambda (t_{n+1}-\tau)}(t_{n+1}-\tau)^{\alpha-1}g(\tau) d\tau
-\sum_{i=0}^{m_n-1}\int_{\tau_{i,n}}^{\tau_{i+1,n}}
e^{-\lambda (t_{n}-\tau)}(t_{n}-\tau)^{\alpha-1}g(\tau) d\tau
\nonumber\\
&\approx&\sum_{i=0}^{m_n-1}\int_{\tau_{i,n}}^{\tau_{i+1,n}} (t_{n+1}-\tau)^{\alpha-1}
\frac{e^{-\lambda (t_{n+1}-\tau_{i+1,n})}g(\tau_{i+1,n})(\tau-\tau_{i,n})
+e^{-\lambda (t_{n+1}-\tau_{i,n})}g(\tau_{i,n})(\tau_{i+1,n}-\tau)}{\tau_{i+1,n}-\tau_{i,n}}
d\tau
\nonumber\\
&&-\sum_{i=0}^{m_n-1}\int_{\tau_{i,n}}^{\tau_{i+1,n}} (t_{n}-\tau)^{\alpha-1}
\frac{e^{-\lambda (t_{n}-\tau_{i+1,n})}g(\tau_{i+1,n})(\tau-\tau_{i,n})
+e^{-\lambda (t_{n}-\tau_{i,n})}g(\tau_{i,n})(\tau_{i+1,n}-\tau)}{\tau_{i+1,n}-\tau_{i,n}}
d\tau
\nonumber\\
&=&\frac{h^\alpha}{\alpha(\alpha+1)}\sum_{i=0}^{m_n} (a_{i,n+1}-c_{i,n+1})g(t_{n_{i}}),
\end{eqnarray}
where $\{a_{i,n+1}\}_{i=0}^{m_n}$ are defined in (\ref{e2dot17}), and the definitions of $c_{i,n+1}$, $i=0,1,\cdots,m_n$, are obtained by replacing $n+1$ of (\ref{e2dot17}) corresponding to $a_{i,n+1}$, $i=0,1,\cdots,m_n$, with $n$.
However, it should be noted that $c_{i,n+1}\neq a_{i,n}$, because $n_k\neq (n-1)_k$ for $k=1,2,\cdots$.

Assume that we have calculated the approximations $x_j\approx x(t_j)$, $j=1,2,\cdots,n$. Then by combining with the one step approximations (\ref{e2dot4}) and (\ref{e2dot5}), the predictor-corrector approach to compute $x_{n+1}\approx x(t_{n+1})$ for $\alpha\in(0,2)$ can be yielded as

{\bf Case 1 ($n=0$):}
\begin{equation}\label{e2dot33}
\left\{
\begin{array}{l}
\displaystyle x_1^{Pr}=x_0(t_1)+\frac{h^\alpha}{\Gamma(\alpha+1)} e^{-\lambda h}f(0,x_0), \\
 \\
\displaystyle x_1=x_0(t_1)+\frac{h^\alpha}{\Gamma(\alpha+2)}\left[f(h,x_1^{Pr})
+\alpha e^{-\lambda h} f(0,x_0)\right];
\end{array}
\right.
\end{equation}

{\bf Case 2 ($n\geq 1$):}
\begin{equation}\label{e2dot34}
\left\{
\begin{array}{ll}
x_{n+1}^{Pr} & =  \displaystyle x_0(t_{n+1})+x_n-x_0(t_n)+\frac{h^\alpha}{\Gamma(\alpha+2)}\cdot\\
\displaystyle &
\left[\sum\limits_{i=0}^{m_n-1} (a_{i,n+1}-c_{i,n+1})f(t_i,x_{n_i})
+\left[a_{m_n,n+1}-c_{m_n,n+1}+(\alpha+1)e^{-\lambda h}\right] f(t_n,x_n)\right], \\
\\
x_{n+1} & =\displaystyle x_0(t_{n+1})+x_n-x_0(t_n)+\frac{h^\alpha}{\Gamma(\alpha+2)}\cdot\\
\displaystyle &
\left[\sum\limits_{i=0}^{m_n-1} (a_{i,n+1}-c_{i,n+1})f(t_i,x_{n_i})
+(a_{m_n,n+1}-c_{m_n,n+1}+\alpha e^{-\lambda h}) f(t_n,x_n)
+f(t_{n+1},x_{n+1}^{Pr})\right].
\end{array}
\right.
\end{equation}

\section{Numerical Examples}\label{sec:3}

\begin{example}\label{example:3.1}
Our first example deals with the case that the unknown solution $x(t)$ has a smooth derivative of order $\alpha$. Specifically, we shall consider the
following equation as in \cite{Diethelm:02a,Diethelm:04}:
\begin{equation}\label{equa3.1}
_{0}^{C}D_{t}^{\alpha,\lambda} x(t)=
e^{-\lambda t}\left[\frac{\Gamma(9)}{\Gamma{(9-\alpha)}}t^{8-\alpha}-3\frac{\Gamma{(5+\alpha/2)}}{\Gamma{(5-\alpha/2)}} t^{4-\alpha/2}+
\frac{9}{4}\Gamma{(\alpha+1)}+\big(\frac{3}{2}t^{\alpha/2}-t^4\big)^3
-(e^{\lambda t}x)^{3/2}\right],
\end{equation}
with the initial condition(s) $x(0)=0$ (and $\left[e^{\lambda t}x(t)\right]'\big|_{t=0}=0$ if $1<
\alpha \leq 2 $).
\end{example}

The exact solution of this initial value problem is
\begin{equation}\label{equa3.2}
x(t)=e^{-\lambda t}\left(t^{8}-3t^{4+\alpha/2}+\frac{9}{4}t^\alpha\right),
\end{equation}
so, the right function $_{0}^{C}D_{t}^{\alpha,\lambda} x(t)=\frac{\Gamma(9)}{\Gamma{(9-\alpha)}}t^{8-\alpha}-3\frac{\Gamma{(5+\alpha/2)}}{\Gamma{(5-\alpha/2)}} t^{4-\alpha/2}+\frac{9}{4}\Gamma{(\alpha+1)}\in C^2[0,T]$ for arbitrary $T>0$ and $\alpha\in(0,2)$.

Table \ref{table3.1} verifies that, by choosing suitable parameters $\Delta y$ or $\Delta s$ for different $\alpha\in(0,1]$, the equidistributing predictor-corrector methods can reach to the convergence order $\min\{1+2\alpha,2\}$ as given in Theorem 1. Also, it can be seen that when the time interval is not very large ($T=1$) and $\lambda=1$, the equal-area distribution predictor-corrector methods have already shown their benefits in computation cost for $\alpha\in(0,1]$. As for the case of $\alpha\in(1,2)$, both equidistributing methods show their huge superiority in the aspects of numerical accuracy as well as computation cost. Noting that to compute $x_{n+1}$, the least three required quadrature nodes are $0,~t_{n}$, and $t_{n+1}$,  except when $n=0$, only two points $0$ and $t_1$ are employed. So, given $h$ and $T$, the total times of the quadrature nodes being used in the procedure is $2+3(\frac{T}{h}-1)$, i.e., $29$, $59$, $119$, $239$, and $479$, respectively, when $h=1/10$, $1/20$, $1/40$, $1/80$, $1/160$, and $T=1$, which coincides with the values of $M$ in Table \ref{table3.1}; while  the total times of the quadrature nodes being used for nonequidistributing scheme (\ref{e2dot15}) and (\ref{e2dot199})-(\ref{e2dot200}) are $65$, $230$, $860$, $3320$, and $13040$, separately. That means, when $\alpha\in(1,2)$, while keeping high numerical accuracy, the computation cost of the equidistributing methods are linearly increasing with time, comparing with the $O(h^{-2})$ expenditure in the predictor-corrector methods of \cite{Diethelm:02a,Deng:07d}.
\begin{remark}
For not losing accuracy, the basic strategy of choosing the values of $\Delta y$ or $\Delta s$ is: 1) for the algorithms generated from (\ref{e2dot2}) where $\alpha\in(0,1)$, if $\alpha$ is close to $1$, $\Delta y$ or $\Delta s$ is approximately equal to $h$, and if $\alpha$ is close to $0$, $\Delta y$ or $\Delta s$ can be bigger than $h$, say, $10h$; 2) for the algorithms generated from (\ref{e2dot3}) where $\alpha \in (0,2)$,  when $\alpha$ is closer to $1$, $\Delta y$ or $\Delta s$ should be smaller (much smaller than $h$ if $\alpha\in (0,1)$ and around $h$ if $\alpha \in (1,2)$).
\end{remark}
\begin{table}\small
\caption{The maximum errors ($e_{max}$), convergence rates (CO), and the total times of quadrature nodes being used ($M$) of Example \ref{example:3.1} at $T=1$ and $\lambda=1$,
by using the equal-height distribution as well as the equal-area distribution ones,
for different $\alpha$, $h$, $\Delta y$, and $\Delta s$, respectively.}\label{table3.1}
\begin{tabular}{ccccccccccc}
\hline\noalign{\smallskip}
\multicolumn{2}{c}{$\alpha$}  &\multicolumn{3}{c}{$\alpha=0.2~(\Delta y=\Delta s=10 h)$} &\multicolumn{3}{c}{$\alpha=0.5~(\Delta y=\Delta s=\frac{5h}{2})$}
&\multicolumn{3}{c}{$\alpha=0.8~(\Delta y=\Delta s=h)$}\\
\noalign{\smallskip}\hline\noalign{\smallskip}
Method & $h$ & $e_{max}$ & CO & M & $e_{max}$ & CO & M & $e_{max}$ & CO & M\\
\noalign{\smallskip}\hline\noalign{\smallskip}
Equal-      &1/10 &5.83 1e-1 &-    &32   &2.87 1e-2 &-    &57   &8.23 1e-3 &-    &64    \\
height      &1/20 &2.01 1e-1 &1.52 &148  &6.29 1e-3 &2.19 &205  &1.62 1e-3 &2.35 &228  \\
            &1/40 &4.04 1e-2 &2.32 &588  &1.44 1e-3 &2.12 &777  &3.44 1e-4 &2.23 &854  \\
of (2.2)     &1/80 &9.59 1e-3 &2.08 &2327 &3.42 1e-4 &2.08 &3034 &8.22 1e-5 &2.07 &3300 \\
            &1/160&2.78 1e-3 &1.78 &9214 &8.24 1e-5 &2.05 &11938&1.98 1e-5 &2.06 &12976\\
\noalign{\smallskip}\hline\noalign{\smallskip}
Equal-      &1/10 &5.84 1e-1 &-    &29   &3.29 1e-2 &-    &43   &8.23 1e-3 &-    &61    \\
area        &1/20 &2.00 1e-1 &1.55 &59   &6.80 1e-3 &2.28 &160  &1.62 1e-3 &2.35 &221  \\
            &1/40 &4.07 1e-2 &2.30 &287  &1.56 1e-3 &2.12 &609  &3.34 1e-4 &2.27 &830  \\
of (2.2)    &1/80 &9.68 1e-3 &2.07 &1206 &3.64 1e-4 &2.10 &2379 &7.23 1e-5 &2.21 &3210 \\
            &1/160&2.81 1e-3 &1.79 &4885 &8.74 1e-5 &2.06 &9390 &1.73 1e-5 &2.07 &12620\\
\midrule
\midrule
\multicolumn{2}{c}{$\alpha$}  &\multicolumn{3}{c}{$\alpha=0.2~(\Delta y=\frac{h}{2},~\Delta s=h)$} &\multicolumn{3}{c}{$\alpha=0.5~(\Delta y=\frac{h}{10},~\Delta s=\frac{h}{5})$}
&\multicolumn{3}{c}{$\alpha=0.8~(\Delta y=\frac{h}{50},~\Delta s=\frac{h}{16})$}\\
\noalign{\smallskip}\hline\noalign{\smallskip}
Equal-      &1/10 &5.70 1e-1 &-    &63   &2.86 1e-2 &-    &65   &8.23 1e-3 &-    &65    \\
height      &1/20 &2.01 1e-1 &1.50 &208  &6.29 1e-3 &2.19 &230  &1.62 1e-3 &2.35 &230  \\
            &1/40 &4.04 1e-2 &2.32 &700  &1.44 1e-3 &2.12 &828  &3.49 1e-4 &2.21 &860  \\
of (2.3)     &1/80 &9.59 1e-3 &2.08 &2370 &3.43 1e-4 &2.07 &2899 &8.03 1e-5 &2.12 &3278 \\
            &1/160&2.79 1e-3 &1.78 &7887 &8.71 1e-5 &1.98 &9752 &1.67 1e-5 &2.26 &11521\\
\noalign{\smallskip}\hline\noalign{\smallskip}
Equal-      &1/10 &5.81 1e-1 &-    &35   &2.44 1e-2 &-    &65   &7.77 1e-3 &-    &69    \\
area        &1/20 &2.00 1e-1 &1.54 &86   &6.38 1e-3 &1.93 &182  &1.44 1e-3 &2.43 &223  \\
            &1/40 &6.50 1e-2 &1.62 &235  &2.82 1e-3 &1.18 &529  &6.00 1e-4 &1.26 &657  \\
of (2.3)    &1/80 &8.35 1e-3 &2.96 &658  &3.84 1e-4 &2.87 &1523 &1.52 1e-4 &1.98 &1757 \\
            &1/160&4.72 1e-4 &4.15 &1903 &2.23 1e-5 &4.11 &4112 &1.54 1e-5 &3.31 &4315\\
\midrule
\midrule
\multicolumn{2}{c}{$\alpha$}  &\multicolumn{3}{c}{$\alpha=1.2~(\Delta y=\Delta s=10 h)$} &\multicolumn{3}{c}{$\alpha=1.5~(\Delta y=\Delta s=10 h)$}
&\multicolumn{3}{c}{$\alpha=1.8~(\Delta y=\Delta s=10 h)$}\\
\noalign{\smallskip}\hline\noalign{\smallskip}
Method & $h$ & $e_{max}$ & CO & M & $e_{max}$ & CO & M & $e_{max}$ & CO & M\\
\noalign{\smallskip}\hline\noalign{\smallskip}
Equal-       &1/10 &1.04 1e-4 &-    &29   &9.00 1e-5 &-    &29   &7.40 1e-5 &-    &29    \\
height       &1/20 &4.46 1e-6 &4.54 &59   &3.51 1e-6 &4.68 &59   &2.61 1e-6 &4.83 &59  \\
             &1/40 &1.88 1e-7 &4.57 &119  &1.34 1e-7 &4.71 &119  &8.95 1e-8 &4.86 &119  \\
             &1/80 &7.86 1e-9 &4.58 &239  &5.03 1e-9 &4.73 &239  &3.04 1e-9 &4.88 &239 \\
             &1/160&3.26 1e-10&4.59 &479  &1.88 1e-10&4.74 &479  &1.02 1e-10&4.89 &479\\
\noalign{\smallskip}\hline\noalign{\smallskip}
Equal-      &1/10 &1.04 1e-4 &-    &29   &9.00 1e-5 &-    &29   &7.40 1e-5 &-    &31    \\
area        &1/20 &4.46 1e-6 &4.54 &59   &3.51 1e-6 &4.68 &59   &2.61 1e-6 &4.83 &61  \\
            &1/40 &1.88 1e-7 &4.57 &119  &1.34 1e-7 &4.71 &119  &8.95 1e-8 &4.86 &122  \\
            &1/80 &7.86 1e-9 &4.58 &239  &5.03 1e-9 &4.73 &240  &3.04 1e-9 &4.88 &246 \\
            &1/160&3.26 1e-10&4.59 &479  &1.88 1e-10&4.74 &482  &1.02 1e-10&4.89 &493\\
\noalign{\smallskip}\hline
\end{tabular}
\end{table}

\begin{example}\label{example:3.2}
Next we come to the case that the unknown solution $x(t)$ itself is a smooth function, but the given function $f(t,x(t))$ has weak regularity. Specifically, we consider the linear equation:
\begin{equation}\label{equa3.3}
_{0}^{C}D_{t}^{\alpha,\lambda} x(t)=\left\{
\begin{array}{ll}
e^{-\lambda t}\left[\frac{2}{\Gamma{(3-\alpha)}}t^{2-\alpha}-e^{\lambda t}x(t)+t^2-t\right],& ~\textrm{for}~ \alpha>1,\\
e^{-\lambda t}\left[\frac{2}{\Gamma{(3-\alpha)}}t^{2-\alpha}-\frac{1}{\Gamma{(2-\alpha)}}t^{1-\alpha}-e^{\lambda t}x(t)+t^2-t\right],
& ~\textrm{for}~ \alpha\leq 1,
\end{array}
\right.
\end{equation}
with the initial condition(s) $x(0)=0$ (and $\left[e^{\lambda t}x(t)\right]'\big|_{t=0}=-1$ if $1<\alpha \leq 2 $).
\end{example}

The exact solution of this initial value problem is
\begin{equation}\label{equa3.4}
x(t)=e^{-\lambda t}(t^2-t),
\end{equation}
so, the right function
\begin{equation}\label{equa3.5}
_{0}^{C}D_{t}^{\alpha,\lambda} x(t)=\left\{
\begin{array}{ll}
\frac{2}{\Gamma{(3-\alpha)}}t^{2-\alpha},& ~\textrm{for}~ \alpha>1,\\
\frac{2}{\Gamma{(3-\alpha)}}t^{2-\alpha}-\frac{1}{\Gamma{(2-\alpha)}}t^{1-\alpha},& ~\textrm{for}~ \alpha\leq 1,
\end{array}
\right.
\end{equation}
is continuous, but its first order derivative is infinite at $t=0$.

Table \ref{table3.2} shows that when the time interval is larger ($T=5$), the advantage of computation cost for equidistributing schemes becomes more obvious. For $\alpha\in(1,2)$, the proposed methods of this paper possess more obvious advantages; they can even reach to the convergence order of $2$, although the function $f(t,x(t))$ has a weak regularity. Moreover, as discussed in Example \ref{example:3.1},  the total times of the quadrature nodes being used are nearly equal to $2+3(\frac{T}{h}-1)$ for $\alpha\in(1,2)$, which means the computation cost of the equidistributing methods increases linearly with the time evolution. We can have a better view from Fig. \ref{fig:3.1} with $\alpha=0.5$ and Fig. \ref{fig:3.2} with $\alpha=1.5$, where the equidistributing methods manage to linear growth of CPU time for long time interval ($T=50$). In this section, all numerical computations are done in Matlab 7.11.0 on a normal laptop with 1GB of memory.

\begin{table}\small
\caption{The maximum errors ($e_{max}$), convergence rates (CO), and the total times of quadrature nodes being used ($M$) of Example \ref{example:3.2} at $T=5$ and $\lambda=1$,
by using the equal-height distribution as well as the equal-area distribution ones,
for different $\alpha$, $h$, $\Delta y$, and $\Delta s$, respectively.}\label{table3.2}
\begin{tabular}{ccccccccccc}
\hline\noalign{\smallskip}
\multicolumn{2}{c}{$\alpha$}  &\multicolumn{3}{c}{$\alpha=0.2~(\Delta y=\Delta s=\frac{h}{2})$} &\multicolumn{3}{c}{$\alpha=0.5~(\Delta y=\Delta s=\frac{h}{2})$}
&\multicolumn{3}{c}{$\alpha=0.8~(\Delta y=\Delta s=\frac{h}{2})$}\\
\noalign{\smallskip}\hline\noalign{\smallskip}
Method & $h$ & $e_{max}$ & CO & M & $e_{max}$ & CO & M & $e_{max}$ & CO & M\\
\noalign{\smallskip}\hline\noalign{\smallskip}
Equal-      &1/10 &4.36 1e-2 &-    &458   &2.25 1e-2 &-    &506   &5.60 1e-2 &-    &571    \\
height      &1/20 &2.11 1e-2 &1.05 &2137  &1.47 1e-2 &0.61 &2472  &1.44 1e-2 &1.96 &2790  \\
            &1/40 &9.64 1e-3 &1.13 &10628 &7.16 1e-3 &1.04 &11758 &7.85 1e-3 &0.88 &12004 \\
of (2.2)    &1/80 &4.28 1e-3 &1.17 &47104 &1.36 1e-3 &2.40 &47721 &4.12 1e-3 &0.93 &48101\\
            &1/160&1.87 1e-3 &1.20 &189260&5.77 1e-4 &1.23 &191072&2.12 1e-3 &0.96 &192220\\
\noalign{\smallskip}\hline\noalign{\smallskip}
Equal-      &1/10 &4.36 1e-2 &-    &397   &2.40 1e-2 &-    &454   &2.48 1e-2 &-    &543    \\
area        &1/20 &2.11 1e-2 &1.05 &1856  &1.55 1e-2 &0.63 &2209  &1.44 1e-2 &0.78 &2651  \\
            &1/40 &9.64 1e-3 &1.13 &9070  &8.96 1e-3 &0.79 &10651 &7.85 1e-3 &0.88 &11513 \\
of (2.2)    &1/80 &4.28 1e-3 &1.17 &41542 &1.67 1e-3 &2.42 &43778 &4.12 1e-3 &0.93 &46100\\
            &1/160&1.87 1e-3 &1.20 &167412&5.77 1e-4 &1.54 &175082&2.12 1e-3 &0.96 &184222\\
\midrule
\midrule
\multicolumn{2}{c}{$\alpha$}  &\multicolumn{3}{c}{$\alpha=0.2~(\Delta y=\frac{h}{160},~\Delta s=\frac{h}{4})$} &\multicolumn{3}{c}{$\alpha=0.5~(\Delta y=\frac{h}{250},~\Delta s=\frac{h}{16})$}
&\multicolumn{3}{c}{$\alpha=0.8~(\Delta y=\frac{h}{120},~\Delta s=\frac{h}{40})$}\\
\noalign{\smallskip}\hline\noalign{\smallskip}
Equal-      &1/10 &4.36 1e-2 &-    &1159  &5.18 1e-3 &-    &1243  &2.48 1e-2 &-    &1181    \\
height      &1/20 &2.11 1e-2 &1.05 &4193  &2.01 1e-3 &1.36 &4538  &1.44 1e-2 &0.78 &4173  \\
            &1/40 &9.64 1e-3 &1.13 &19469 &1.18 1e-3 &0.77 &16412 &7.85 1e-3 &0.88 &14233 \\
of (2.3)    &1/80 &4.28 1e-3 &1.17 &52550 &8.61 1e-4 &0.45 &57446 &4.12 1e-3 &0.93 &46425\\
            &1/160&1.87 1e-3 &1.20 &180836&5.77 1e-4 &0.58 &196538&2.12 1e-3 &0.96 &142949\\
\noalign{\smallskip}\hline\noalign{\smallskip}
Equal-      &1/10 &4.36 1e-2 &-    &223   &5.14 1e-3 &-    &471   &2.48 1e-2 &-    &884    \\
area        &1/20 &2.11 1e-2 &1.05 &585   &2.01 1e-3 &1.36 &1471  &1.44 1e-2 &0.78 &2921  \\
            &1/40 &9.64 1e-3 &1.13 &1602  &1.42 1e-3 &0.51 &4344  &7.85 1e-3 &0.88 &8716 \\
of (2.3)    &1/80 &4.28 1e-3 &1.17 &4394  &8.61 1e-4 &0.71 &12096 &4.12 1e-3 &0.93 &23458\\
            &1/160&1.87 1e-3 &1.20 &12036 &5.77 1e-4 &0.58 &32807 &2.12 1e-3 &0.96 &57378\\
\midrule
\midrule
\multicolumn{2}{c}{$\alpha$}  &\multicolumn{3}{c}{$\alpha=1.2~(\Delta y=\Delta s=10 h)$} &\multicolumn{3}{c}{$\alpha=1.5~(\Delta y=\Delta s=10 h)$}
&\multicolumn{3}{c}{$\alpha=1.8~(\Delta y=\Delta s=10 h)$}\\
\noalign{\smallskip}\hline\noalign{\smallskip}
Method & $h$ & $e_{max}$ & CO & M & $e_{max}$ & CO & M & $e_{max}$ & CO & M\\
\noalign{\smallskip}\hline\noalign{\smallskip}
Equal-      &1/10 &7.97 1e-4 &-    &149   &2.82 1e-3 &-    &149   &4.82 1e-3 &-    &149    \\
height      &1/20 &2.44 1e-4 &1.71 &299   &7.55 1e-4 &1.90 &299   &1.27 1e-3 &1.92 &299 \\
            &1/40 &6.66 1e-5 &1.88 &599   &1.95 1e-4 &1.95 &599   &3.27 1e-4 &1.96 &599  \\
            &1/80 &1.73 1e-5 &1.95 &1199  &4.95 1e-5 &1.98 &1199  &8.27 1e-5 &1.98 &1199 \\
            &1/160&4.39 1e-6 &1.98 &2399  &1.25 1e-5 &1.99 &2399  &2.08 1e-5 &1.99 &2399\\
\noalign{\smallskip}\hline\noalign{\smallskip}
Equal-      &1/10 &7.97 1e-4 &-    &156   &2.82 1e-3 &-    &150   &4.82 1e-3 &-    &155    \\
area        &1/20 &2.44 1e-4 &1.71 &314   &7.56 1e-4 &1.90 &303   &1.27 1e-3 &1.92 &312 \\
            &1/40 &6.66 1e-5 &1.88 &628   &1.95 1e-4 &1.95 &609   &3.27 1e-4 &1.96 &627  \\
            &1/80 &1.73 1e-5 &1.95 &1256  &4.95 1e-5 &1.98 &1221  &8.27 1e-5 &1.98 &1256 \\
            &1/160&4.39 1e-6 &1.98 &2514  &1.25 1e-5 &1.99 &2445  &2.08 1e-5 &1.99 &2513\\
\noalign{\smallskip}\hline
\end{tabular}
\end{table}
\begin{figure}[!htbp]
\includegraphics[scale=0.4]{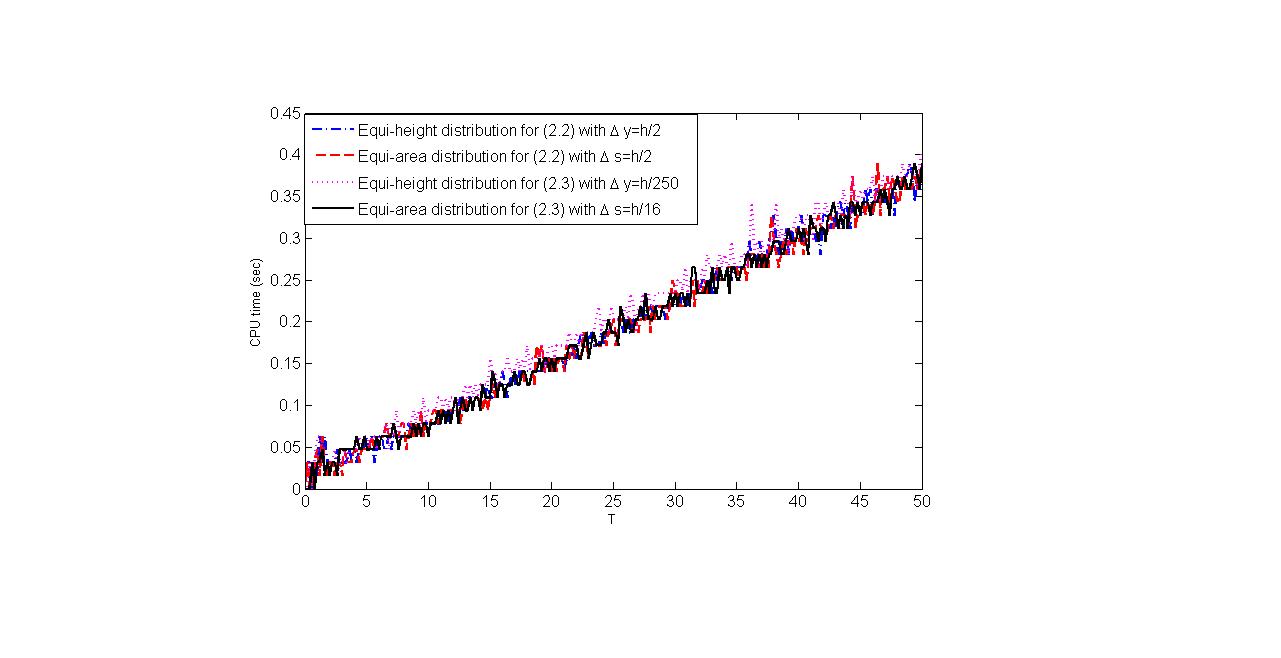}\\
\caption{The CUP time of the equal-height distribution as well as the equal-area distribution ones,
for $\alpha=0.5$, $T=50$, $\lambda=1$, $h=1/20$, and $\Delta y=\Delta s=h/2$.}\label{fig:3.1}
\end{figure}
\begin{figure}[!htbp]
\includegraphics[scale=0.4]{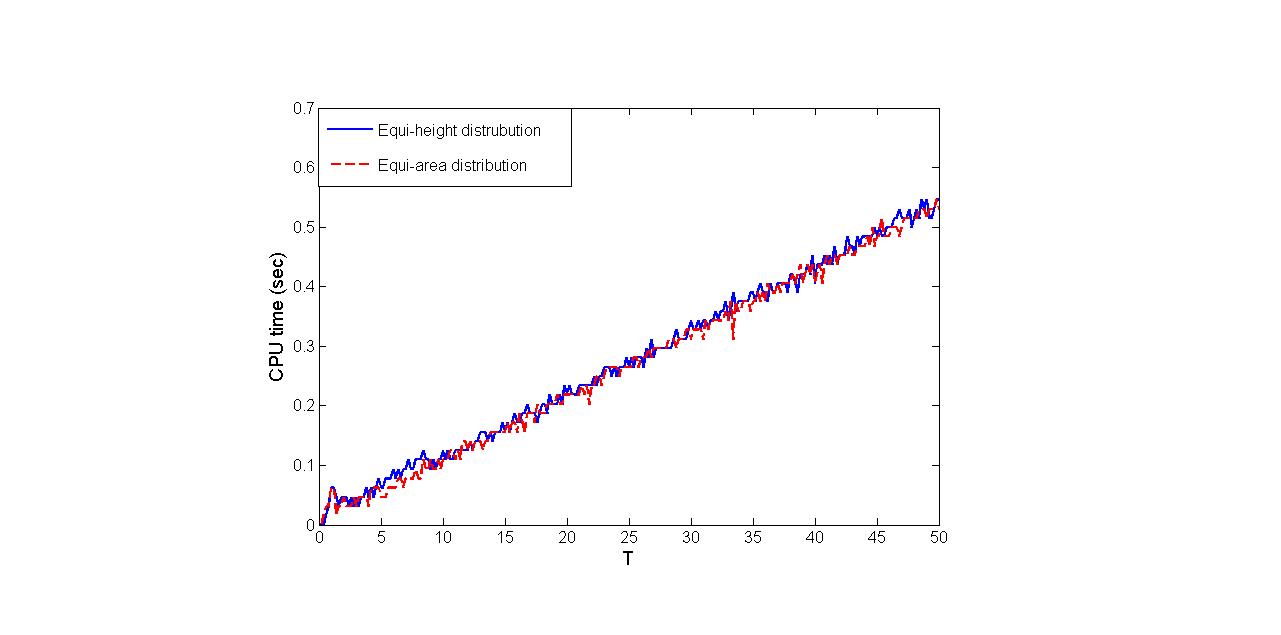}\\
\caption{The CUP time of the equal-height distribution as well as the equal-area distribution ones,
for $\alpha=1.5$, $T=50$, $\lambda=1$, $h=1/20$, and $\Delta y=\Delta s=10h$.}\label{fig:3.2}
\end{figure}
\begin{example}\label{example:3.3}
In this example, we examine the following initial value problem
\begin{equation}\label{equa3.6}
_{0}^{C}D_{t}^{\alpha,\lambda}x(t)=-x(t).
\end{equation}
The initial values are given as $x(0)=1$ (and $[e^{\lambda t}x(t)]'\big|_{t=0}=0$ if $1<\alpha<2$).
It is known that \cite{Li:14} the exact solution of this initial value problem is
$$x(t)=e^{-\lambda t}E_{\alpha,1}(-t^{\alpha}).$$
Here the generalized Mittag-Leffler function $E_{\alpha,\beta}(\cdot)$ is given by \cite{Podlubny:99}
\begin{equation}\label{equa3.7}
E_{\alpha,\beta}(z)
=\sum_{k=0}^{\infty}\frac{z^k}{\Gamma(\alpha k+\beta)},\quad \text{Re}(\alpha)>0.
\end{equation}
\end{example}

It is obvious that neither $x(t)$ nor $_{0}^{C}D_{t}^{\alpha,\lambda}x(t)$
has a bounded first (second) derivative at $t=0$ when $0<\alpha\leq 1$ ($1<\alpha\leq 2$).
Table \ref{table3.3} shows the numerical accuracy of the equidistributing methods by choosing suitable $\Delta y$ and $\Delta s$ for different $\alpha\in (0,1)$ when $T=4$ and $\lambda=1$. And the equal-area distribution methods still show their advantages in computation cost. 
\begin{table}\small
\caption{The maximum errors ($e_{max}$), convergence rates (CO), and the total times of quadrature nodes being used ($M$) of Example \ref{example:3.3} at $T=4$ and $\lambda=1$,
by using the equal-height distribution as well as the equal-area distribution ones,
for different $\alpha$, $h$, $\Delta y$, and $\Delta s$, respectively.}\label{table3.3}
\begin{tabular}{ccccccccccc}
\hline\noalign{\smallskip}
\multicolumn{2}{c}{$\alpha$}  &\multicolumn{3}{c}{$\alpha=0.2~(\Delta y=\Delta s= 10h)$} &\multicolumn{3}{c}{$\alpha=0.5~(\Delta y=\Delta s=h)$}
&\multicolumn{3}{c}{$\alpha=0.8~(\Delta y=\Delta s=h)$}\\
\noalign{\smallskip}\hline\noalign{\smallskip}
Method & $h$ & $e_{max}$ & CO & M & $e_{max}$ & CO & M & $e_{max}$ & CO & M\\
\noalign{\smallskip}\hline\noalign{\smallskip}
Equal-      &1/10 &1.08 1e-1 &-    &122   &6.54 1e-3 &-    &314   &1.07 1e-2 &-    &331    \\
height      &1/20 &7.53 1e-2 &0.52 &337   &3.58 1e-3 &0.87 &1570  &5.24 1e-3 &1.23 &1660  \\
            &1/40 &5.07 1e-2 &0.57 &1404  &1.72 1e-3 &1.06 &7189  &1.95 1e-3 &1.43 &6877 \\
of (2.2)    &1/80 &3.31 1e-2 &0.61 &7214  &9.00 1e-4 &0.93 &28875 &5.49 1e-4 &1.83 &27372\\
            &1/160&2.11 1e-2 &0.65 &37504 &5.88 1e-4 &0.61 &115585&1.53 1e-4 &1.84 &109480\\
\noalign{\smallskip}\hline\noalign{\smallskip}
Equal-      &1/10 &1.08 1e-1 &-    &119   &7.64 1e-3 &-    &269   &1.07 1e-2 &-    &295    \\
area        &1/20 &7.53 1e-2 &0.52 &296   &4.05 1e-3 &0.91 &1295  &5.46 1e-3 &0.97 &1500  \\
            &1/40 &5.07 1e-2 &0.57 &1073  &2.00 1e-3 &1.02 &6162  &2.17 1e-3 &1.33 &6221 \\
of (2.2)    &1/80 &3.31 1e-2 &0.61 &5015 &9.00 1e-4 &1.15 &24833 &6.12 1e-4 &1.83 &24944\\
            &1/160&2.11 1e-2 &0.65 &23817 &5.88 1e-4 &0.61 &99550 &1.68 1e-4 &1.86 &99922\\
\midrule
\midrule
\multicolumn{2}{c}{$\alpha$}  &\multicolumn{3}{c}{$\alpha=0.2~(\Delta y=\frac{h}{40},~\Delta s= h)$} &\multicolumn{3}{c}{$\alpha=0.5~(\Delta y=\frac{h}{80},~\Delta s=\frac{h}{10})$}
&\multicolumn{3}{c}{$\alpha=0.8~(\Delta y=\frac{h}{400},~\Delta s=\frac{h}{40})$}\\
\noalign{\smallskip}\hline\noalign{\smallskip}
Equal-      &1/10 &1.08 1e-1 &-    &692   &4.05 1e-3 &-    &314   &6.09 1e-4 &-    &860    \\
height      &1/20 &7.53 1e-2 &0.52 &2437  &1.80 1e-3 &1.17 &1570  &2.81 1e-4 &1.12 &3295  \\
            &1/40 &5.07 1e-2 &0.57 &8468  &1.20 1e-3 &0.59 &7189  &1.15 1e-4 &1.29 &12458 \\
of (2.3)    &1/80 &3.31 1e-2 &0.61 &2896  &9.00 1e-4 &0.41 &28875 &4.37 1e-5 &1.40 &45625\\
            &1/160&2.11 1e-2 &0.65 &96830 &5.88 1e-4 &0.61 &115585&2.61 1e-5 &0.74 &15926\\
\noalign{\smallskip}\hline\noalign{\smallskip}
Equal-      &1/10 &1.08 1e-1 &-    &125   &4.05 1e-3 &-    &471   &6.09 1e-4 &-    &707    \\
area        &1/20 &7.53 1e-2 &0.52 &266   &1.80 1e-3 &1.17 &1471  &2.81 1e-4 &1.12 &2333  \\
            &1/40 &5.07 1e-2 &0.57 &595   &1.20 1e-3 &0.59 &4344  &1.15 1e-4 &1.29 &6994 \\
of (2.3)    &1/80 &3.31 1e-2 &0.61 &1378  &9.00 1e-4 &0.41 &12096 &5.05 1e-5 &1.19 &18771\\
            &1/160&2.11 1e-2 &0.65 &3343  &5.88 1e-4 &0.61 &32807 &2.38 1e-5 &1.09 &45730\\
\noalign{\smallskip}\hline
\end{tabular}
\end{table}

\section{Conclusions}\label{sec:4}
For more deeply discussing the anomalous diffusion, sometimes the L\'evy waiting time distribution has to be tempered because of the finite lifetime of biological particles. Then the time derivative of the Fokker-Planck equation describing this type of trapped dynamics is the tempered fractional operator. Since the tempered fractional operator is still nonlocal, generally the computation cost of numerically solving the tempered time-dependent problem is quadratically increasing with time $t$. This paper provides the predictor-corrector approach with theoretically proved convergence order $\min\{1+2\alpha,2\}$ for the tempered fractional differential equation. The proposed predictor-corrector schemes on equidistributed meshes are detailedly discussed.   In fact, the main contribution of this paper comes from the proposed equidistributing schemes which have linearly increasing computation cost with time $t$ while keeping the accuracy; when $\alpha \in (1,2)$, the numerical results show that the convergence order can  even exceed $4$. And the larger the time is, the more benefits are obtained for the equidistributing methods.


\section*{Acknowledgements}
The authors thank W.H. Deng for the discussions. This work was supported by the National Natural Science Foundation of China under Grant 11271173 and 11471150, the Fundamental Research Funds for the Central Universities (Northwest University for Nationalities) under Grant No. 31920130007, the Humanities and Social Sciences Youth Projects of the Ministry of Education of P R China under Grant No. 12YJCZH027 and No. 13YJCZH029, and the Humanities and Social Sciences Planning Projects of the Ministry of Education of P R China under Grant No. 11YJAZH053.


  \end{document}